\documentclass[a4paper,11pt,reqno]{amsart}

\usepackage[utf8]{inputenc}
\usepackage[T1]{fontenc}
\usepackage{lmodern}
\usepackage[english]{babel}
\usepackage{microtype}
\usepackage{amsmath,amssymb,amsfonts,amsthm}
\usepackage{mathtools}
\usepackage{aliascnt}
\usepackage{mathrsfs}
\usepackage{braket}
\usepackage{bm}

\usepackage[a4paper,hmargin=3.25cm,bottom=4cm,top=3.5cm,footskip=3\baselineskip]{geometry}
\usepackage[colorlinks,citecolor=blue]{hyperref}

\usepackage{tikz}

\usepackage{enumerate}

\makeatletter
\def\newaliasedtheorem#1[#2]#3{
  \newaliascnt{#1@alt}{#2}
  \newtheorem{#1}[#1@alt]{#3}
  \expandafter\newcommand\csname #1@altname\endcsname{#3}
}
\makeatother

\numberwithin{equation}{section}

\theoremstyle{plain}
\newtheorem{theorem}{Theorem}[section]
\newaliasedtheorem{proposition}[theorem]{Proposition}
\newaliasedtheorem{lemma}[theorem]{Lemma}
\newaliasedtheorem{corollary}[theorem]{Corollary}
\newaliasedtheorem{counterexample}[theorem]{Counterexample}

\theoremstyle{definition}
\newaliasedtheorem{definition}[theorem]{Definition}
\newaliasedtheorem{question}[theorem]{Question}
\newaliasedtheorem{example}[theorem]{Example}

\theoremstyle{remark}
\newaliasedtheorem{remark}[theorem]{Remark}

\newcommand{\R}{\mathbb{R}}

\renewcommand{\c}{\mathsf{c}}
\newcommand{\eps}{\varepsilon}
\newcommand{\Id}{\operatorname{Id}}

\newcommand{\C}{C}
\newcommand{\dist}{\mathsf{d}}
\renewcommand{\d}{\mathrm{d}}
\newcommand{\diam}{\operatorname{diam}}

\newcommand{\f}{\mathsf{f}}
\newcommand{\g}{\mathsf{g}}

\newcommand{\sv}{\mathsf{v}}

\newcommand{\cV}{\mathcal{V}}
\newcommand{\cI}{\mathcal{I}}

\newcommand{\jet}{\mathfrak{J}}

\renewcommand{\bra}[1]{\left(#1\right)}

\newcommand{\cur}[1]{\left\{#1\right\}}
\newcommand{\sqa}[1]{\left[#1\right]}
\newcommand{\ang}[1]{\left<#1\right>}

\newcommand{\abs}[1]{\left\lvert#1\right\rvert}
\newcommand{\norm}[1]{\left\lVert#1\right\rVert}
\newcommand{\nnrm}[1]{{\left\vert\kern-0.25ex\left\vert\kern-0.25ex\left\vert #1
    \right\vert\kern-0.25ex\right\vert\kern-0.25ex\right\vert}}

\newcommand{\simp}{\mathfrak{S}}

\newcommand{\res}{\llcorner} 
\newcommand{\diff}{\mathbb{D}}

\begin{document}

\title[On rough exterior differential systems]{On exterior differential systems involving differentials of H\"{o}lder functions}


\author{Eugene Stepanov}
\address{St.Petersburg Branch of the Steklov Mathematical Institute of the Russian Academy of Sciences,
Fontanka 27,
191023 St.Petersburg, Russia
\and
Department of Mathematical Physics, Faculty of Mathematics and Mechanics,
St. Petersburg State University 
\and 
Faculty of Mathematics, Higher School of Economics, Moscow}
\email{stepanov.eugene@gmail.com}

\author{Dario Trevisan}
\address{Dario Trevisan, Dipartimento di Matematica, Universit\`a di Pisa \\
Largo Bruno Pontecorvo 5 \\ I-56127, Pisa}
\email{dario.trevisan@unipi.it}

\thanks{  The work of the first author has been partially financed by the University of Pisa Visiting Fellow program 2021 and by the RFBR grant \#20-01-00630 A.	
The work of the second author has been partially supported by the INdAM GNAMPA project 2020 ``Problemi di ottimizzazione con vincoli via trasporto ottimo e incertezza''.}

\date{\today}

\maketitle

\begin{abstract}
We study the validity of an extension of Frobenius theorem on integral manifolds for some classes of 
Pfaff-type systems 
of partial differential equations
involving multidimensional ``rough'' signals, i.e.\ ``differentials'' of given 
H\"{o}lder continuous functions interpreted in a suitable way, similarly to Young Differential Equations  in Rough Paths theory.
This can be seen as a tool to study solvability
of exterior differential systems involving rough differential forms, 
i.e.\ the forms involving weak (distributional) derivatives of highly irregular
(e.g.\ H\"{o}lder continuous) functions; the solutions (integral manifolds) 
being also some very weakly regular geometric structures.
\end{abstract}

\tableofcontents

\section{Introduction}\label{sec_Intro1}

The basic tool to study solvability of exterior differential systems
is the classical Frobenius theorem, see e.g.~\cite[theorem~VI.3.1]{hartman_ordinary_2002}, providing necessary and sufficient conditions for existence and uniqueness of manifolds tangent to a given family of vector fields, 
or equivalently, from a dual point of view, 
annihilating a given family of smooth differential $1$-forms. 
If the family consists of a single vector field, it is just the 
existence and uniqueness theorem for solutions of ODE's, but in the general
case of many vector fields an extra geometric condition, usually called 
{\em involutivity}, formulated in terms of commutators (Lie brackets) of vector
fields (or, equivalently, in terms of exterior differentials of forms)  appears to be essential for solvability. This condition requires differentiation, and thus the classical Frobenius theorem can only be applied when the data, i.e., vector
fields and/or differential forms, are smooth.

\subsubsection*{A motivating example}
The extension of Frobenius theorem to exterior differential systems 
with low regularity is quite nontrivial even in the case when the data
are just Lipschitz continuous (hence, still differentiable, but only 
almost everywhere); important steps in this direction have been done 
in~\cite{simic_1996},
~\cite{rampazzo_2007} 
and~\cite{montanari_frobenius-type_2013}. 
A further remarkable result is contained in~\cite{luzzatto_integrability_2016}, 
where one considers vector fields generating just continuous distributions of
hyperplanes with some extra conditions.
In this paper we make an attempt to consider some exterior differential systems involving terms which are not functions but weak (distributional) derivatives of H\"{o}lder functions.
As a motivating example, consider the problem of finding a 
surface in $\R^3$ parameterized as a graph of a
function $\theta: I^2 \to \R$ 
(with $I\subset \R$  interval) satisfying the following  Pfaff system of differential equations on $I^2$, 
\begin{equation} \label{eq:pfaff} \begin{cases} \partial_1 \theta (s_1,s_2) & = f(s_1, s_2, \theta (s_1,s_2)) \partial_1 g (s_1,s_2) \\
 \partial_2 \theta(s_1,s_2) &= f(s_1, s_2,\theta(s_1,s_2)) \partial_2 g (s_1,s_2),
\end{cases}\end{equation}
with $\partial_i = \partial_{s_i}$, $i \in \cur{1,2}$, $f\colon I^2 \times \R \to \R$, $g\colon I^2 \to \R$ smooth.
Since a smooth $\theta\colon I^2\to \R$ necessarily satisfies $\partial_2 \partial_{1} \theta = \partial_{1} \partial_2 \theta$, then a 
straightforward computation gives 
for a classical solution $\theta$ of~\eqref{eq:pfaff}
the following version of the involutivity condition
\begin{equation}\label{eq:involutivity-intro} \partial_2 f \partial_1 g = \partial_1f \partial_2g,
\end{equation}
that should be valid at every point $(s_1, s_2,\theta(s_1,s_2))\in I^2\times \R$ with $(s_1,s_2) \in I^2$. 
Frobenius theorem ensures that the above necessary condition~\eqref{eq:involutivity-intro} is also sufficient for 
unique solvability of~\eqref{eq:pfaff} when $f$ and $g$ 
are sufficiently smooth: namely, if~\eqref{eq:involutivity-intro} holds for every $(x,y,v)\in I^2\times \R$, with $f$ evaluated at $(x, y,v)$ and $g$ at $(x,y)$, then~\eqref{eq:pfaff} admits locally
a unique solution
with the prescribed ``initial'' condition $\theta(s_0, t_0) = \theta_0$, for  $(s_0, t_0) \in I^2$, $\theta_0\in \R$.

We are aimed at solving~\eqref{eq:pfaff} and similar  systems
when $g$ is only H\"{o}lder continuous, say $g \in C^{\beta}(I^2)$, in a robust way, e.g., stable with respect to approximation of $g$ by sequences 
in $C^{\beta}(I^2)$. Since $g$ may be nowhere differentiable, even the definition of a solution 
to~\eqref{eq:pfaff} is not immediate (in fact, formally 
as written~\eqref{eq:pfaff} makes no sense) Moreover, one has to understand what is the
correct version of the solvability condition which should reduce 
to~\eqref{eq:involutivity-intro} in case all the data are smooth. 
This fits naturally in the general direction of research 
receiving growing attention nowadays, related 
to differential equations with potentially 
purely non-differentiable (like H\"older continuous) unknowns, 
and usually referred to as {\em ``Rough paths theory''} 
or just {\em ``Rough calculus''}~\cite{friz_course_2014, gubinelli_controlling_2004}.
It is worth observing that if $\theta$, $g$ 
and $f(\cdot, v)$ for every $v\in \R$ depend only 
on one variable instead of two, then~\eqref{eq:pfaff} (which becomes then a single equation instead of a system) can be thought as a \textit{Rough Differential Equation} (RDE), a rough analogue of an ODE, and so in general situation~\eqref{eq:pfaff} can be considered a rough analogue of the Pfaff system of PDEs.

The original raison 
d'\^{e}tre of the rough calculus is stochastics; 
in fact, rough paths and rough differential equations originally came in a certain sense as an alternative to classical stochastic differential equations allowing to study each  trajectory of the stochastic flow without referring to any underlying
martingale structure. In particular, one can think of~\eqref{eq:pfaff} (to be correctly interpreted) when $g$ is a representative of some stochastic function 
(e.g.~a Brownian sheet).
However, recently, a growing number of 
applications have widened the scope of rough calculus beyond stochastics
to include purely geometric problems. For instance, in~\cite{magnani_2018} a rough calculus approach has been shown natural to tackle a particular case of a well-known problem of subriemannian geometry of graded nilpotent Lie groups, 
namely, the study of the structure of level sets of maps only intrinsically differentiable
in the sense of P.~Pansu~\cite{pansu_1989} (such maps are known to be
generically irregular in the Eucliedan sense). Namely, it has been shown that level sets of maps from the Heisenberg group $H^1$ to $\R^2$, regular only in the intrinsic sense of $H^1$, are curves, possibly only H\"{o}lder regular, satisfying some ``autonomous'' analogue of
an RDE, called in this case Level Set Differential Equation. Extending this approach to maps over
higher order Carnot groups, one naturally concludes that their level sets should satisfy some rough analogues of Pfaff system of PDEs, e.g.\ 
similar to~\eqref{eq:pfaff}, but with $g$ depending on the unknown $\theta$ rather than on coordinates only, thus presently outside the scope of the theory developed in this work.

\subsubsection*{Rough exterior differential systems}
Indeed, our aim is to prove some Frobenius-type solvability
results for rough analogues
of Pfaff equations similar to~\eqref{eq:pfaff}, where irregularity is due
to the presence of ``differentials'' of given H\"{o}lder continuous functions
(such as $g$ in~\eqref{eq:pfaff}), which do not depend on the unknown.      
Such systems of equations can be naturally interpreted as exterior differential systems provided by a set of ``rough differential $1$-forms'', a notion that has been developed and studied in~\cite{stepanov_sewing_17}. 
For instance, to write~\eqref{eq:pfaff} in the language of forms, one may consider the $1$-form in $\R^3$ written $\omega := \, \d x_3- f \, \d g$,
where $f=f(x_1,x_2, x_3)$, $g=g(x_1,x_2)$ and
$(x_1, x_2, x_2)$ denote coordinates in $\R^3$, 
and defined by its action $\ang{[pq],\omega}$ on oriented segments  $[pq]$
by the formula
\[
\ang{[pq],\omega}:=\int_{[pq]}\, dx_3  -\int_{[pq]} f\, \d g =
(q_3-p_3)  -\int_{[pq]} f\, \d g , 
\]
where the integral is understood as the Young integral~\cite{young_inequality_1936} of the restrictions
of $f$ and $g$ to $[pq]$. If we are looking for a $\bar\theta\colon I^2\to \R^3$ such that
\[
\begin{split}
\bar\theta^*\,dx_1 &=ds_1, 
\quad \bar\theta^*\,dx_2 =ds_2,\\
\bar\theta^*\,\omega &=0,
\end{split}
\]  
with $(x_1, x_2, x_2)$ denoting coordinates in $\R^3$
and
$(s_1,s_2)$ denoting coordinates in $I^2\subset \R^2$, 
then  from the first two equations above we have $\bar\theta_1(s_1,s_2) = s_1+c_1$,
$\bar\theta_2(s_1,s_2) = s_2+c_2$, with $c_i$ arbitrary constants, $i=1,2$. Choosing $c_1=c_2=0$, we get
from the third equation that $\theta:=\bar \theta_3$ satisfies
\begin{equation}\label{eq_pf_theta1}
\ang{[ab],\d\theta}= \theta(b)-\theta(a) =
\int_{[ab]} f(s_1, s_2, \theta(s_1,s_2)) \, \d g(s_1,s_2),
\end{equation}
for every $a=(a_1,a_2)$, $b=(b_1,b_2)$ in $I^2$. If $g$ is smooth, and so is $\theta$, this implies
\[
\nabla\theta(s_1,s_2)= f(s_1,s_2, \theta(s_1,s_2))
\nabla g(s_1,s_2),
\]
for every $(s_1,s_2)\in I^2$ that is, we recover~\eqref{eq:pfaff}.

\subsubsection*{Discrete approach: germs and asymptotic expansions}
The integral equation~\eqref{eq_pf_theta1} for $\theta$ is formally defined
either when $g$ is smooth (while $f$ and $\theta$ are, say, just continuous), 
in which case the integral involved is the classical Riemann (or, equivalently, Lebesgue) one, or when both $g$, $f$ and $\theta$ are just 
H\"{o}lder continuous with appropriate H\"{o}lder exponents, when 
the integral involved can be understood as the Young integral \cite{young_inequality_1936}
over the restrictions of the respective functions to $[ab]$. In the latter case, alternatively one may use instead of ``differentials'' $\d\theta$ and $\d g$
the finite differences $\delta\theta$ and $\delta g$ respectively, 
replacing~\eqref{eq_pf_theta1} with the asymptotic expansions
\begin{equation}\label{eq_pf_theta2}
(\delta\theta)_{ab}:=\theta(b)-\theta(a) =
f(a_1, a_2, \theta(a_1,a_2)) \, (\delta g)_{ab} + o(|b-a|)
\end{equation}
for every $a=(a_1,a_2)$, $b=(b_1,b_2)$ in $I^2$, where
 $(\delta g)_{ab}:= g(b)-g(a)$. Note that in the case of smooth $g$ this is equivalent to~\eqref{eq_pf_theta1}.
Using the language introduced in~\cite{stepanov_sewing_17}, this amounts to replacing the ``rough differential form'' $\omega$ by its discrete germ
\[
\eta:= \delta x_3- f \, \delta g
\]
defined by its action over 
 by its action $\ang{[pq],\omega}$ on segments $[pq]$
by the formula
\[
\ang{[pq],\eta}:=
(q_3-p_3)  -f(p)\, (\delta g)_{pq}. 
\]

\subsubsection*{Involutivity condition in discrete form}
Substituting systematically differentials with finite differences and differential (or, equivalently, integral) equations with the appropriate asymptotic expansions, we are also able to reformulate the involutivity condition~\eqref{eq:involutivity-intro} without the use of partial derivatives of $g$. 
Indeed, in the smooth case, we notice that \eqref{eq:involutivity-intro}  is equivalent to the existence of some function that we denote 
$\diff_g f\colon \R^3\to \R$, such that
\[ 
\nabla_{12} f = \diff_g f \nabla g,
\]
where $\nabla_{1,2}$ stands for the gradient in the first two coordinates in $\R^3$.
Using this identity in the first-order Taylor expansions of $f$ and $g$, when 
the latter functions are smooth, one obtains the expansion
\begin{equation}\label{eq_invol_d1}  f(q, v) - f(p,u) = \diff_g f(p,u) \bra{g(q) - g(p)}  + \partial_3f(p,u)(v-u) + o \bra{ \abs{q-p} + \abs{v-u} },
\end{equation}
which is equivalent to~\eqref{eq:involutivity-intro} for smooth data, and
is a natural replacement for~\eqref{eq:involutivity-intro} (with more precise estimate of the asymptotic error term) when
$g$ is only H\"older continuous, 
since it does not involve derivatives of $g$. This condition may be
naturally called $(g, x_3)$-differentiability of $f$ 
(i.e.\ differentiability with respect to $g$ in the first two variables and usual differentiability in the last one), and is in particular
valid when, say, $f(p,z) = F(g(p), z)$ for some smooth $F\colon \R^2\to \R$.

\subsubsection*{Results}

In the first of our main results, Theorem~\ref{thm:frob-1}, 
we prove 
well-posedness (i.e., existence, uniqueness and stability with respect to approximations) of~\eqref{eq:pfaff} 
and its more general analogues
under the involutivity conditions which in case of~\eqref{eq:pfaff}
reduce to~\eqref{eq_invol_d1} for possibly nonsmooth, just H\"{o}lder continuous $g$, for a substantial range of possible  H\"{o}lder exponents of $g$. In particular, it holds true for $g\in C^\beta(\R^2)$, for instance, when
$f(p,z) = F(g(p), z)$ for some $F \in C^{1, \gamma}(\R^2)$ and 
$$\beta(2 + \gamma) > 2.$$
Note however, that although this covers a wide range of possible regularity
of the data in~\eqref{eq:pfaff}, it is not clear whether this is optimal.
In fact, a smooth 
$\theta$ satisfies~\eqref{eq_pf_theta2} (hence~\eqref{eq:pfaff}),
if and only if, for every differentiable curve $\gamma\colon  (-1,1) \to I^2$, the composition $t \mapsto \theta( \gamma(t) )$ satisfies
\begin{equation*}
\label{eq:pfaff-young0}
 \frac{\d}{\d t} (\theta \circ \gamma) (t)  = f( \gamma(t),  (\theta\circ\gamma) (t) )  \frac{\d}{\d t} (g \circ \gamma) (t), \quad \text{for $t \in (-1,1)$,}
\end{equation*}
that is, 
\begin{equation}
\label{eq:pfaff-young}
  (\theta \circ \gamma) (t) = (\theta\circ\gamma) (0) ) + \int _0^t f( \gamma(v),  (\theta\circ\gamma) (s) ) \,\d ( g\circ\gamma)(s)  \quad \text{for $t \in (-1,1)$.} 
\end{equation}
Therefore, when $g$ is only H\"older continuous, we may require for a solution $\theta$ to~\eqref{eq:pfaff} to be a H\"older continuous function such that, for every $\gamma \in C^1( (-1,1); I^2)$ the curve $\theta
\circ\gamma$ solves~\eqref{eq:pfaff-young} in the sense 
of Young~\cite{young_inequality_1936}.
As in the theory of 
Young Differential Equations~\cite{gubinelli_controlling_2004}, if $g \in C^{\beta}(I^2)$, then one may also expect $\theta \in C^{\beta}(I^2)$, 
and a minimal requirement to give meaning to the integral in~\eqref{eq:pfaff-young} would be then $f\in C^\gamma(\R^3)$, with $$\beta(1+\gamma)>1.$$
Theorem~\ref{thm:frob-1} is still far from this threshold, since e.g.\ even in the case $f= f(x_3)$ it requires $\beta(2 + \gamma)>2$. 

A natural question is what happens if the
rough differential form 
$\omega=\d x_3- f_1 \, \d g_1- f_2 \, \d g_2$
for some $f_i\colon I^2 \times \R \to \R$, $g_i\colon I^2 \to \R$, $i=1,2$, 
i.e.\ the respective problem of the type~\eqref{eq:pfaff} contains several
``rough signals'' $g_i$ instead of just a single $g$. 
A step towards this direction is our
second main result, Theorem~\ref{thm:frob-2}, which gives an extension of
the classical Frobenius theorem (with a version of involutivity condition) to a class of such situations with only H\"{o}lder continuous $g_i$ depending each on a different coordinate (i.e.\ in the above terms, say, $g_1=g_1(s_1)$ and $g_2=g_2(s_2)$).  
Then 
the respective Pfaff system of differential equations becomes
\begin{equation} \label{eq:pfaff} \begin{cases} \partial_1 \theta (s_1,s_2) & = f_1(s_1, s_2, \theta (s_1,s_2)) \partial_1 g_1 (s_1) \\
\partial_2 \theta(s_1,s_2) &= f_2(s_1, s_2,\theta(s_1,s_2)) \partial_2 g_2 (s_2)
\end{cases}\end{equation}
(of course, $\partial_1 g_1$ and $\partial_2 g_1$ are just the ordinary derivatives in this case, since $g_1$ and $g_2$ are assumed to depend each on a single variable only).

A side product, which seems to be of some independent interest, is an implicit function 
Theorem~\ref{thm:dini} for possibly nonsmooth $g$-differentiable maps. 

The basic underlying technical tool used is that of integration of H\"older $2$-forms as introduced in~\cite{stepanov_sewing_17}, developing the construction of \cite{zust_integration_2011}, which provides an extension of one-dimensional Young integrals to integrals
of rough differential forms of higher dimensions (in particular, here integrals of
$2$-forms are needed). Note that in order to avoid requesting the reader to be acquainted with the language and tools developed in~\cite{stepanov_sewing_17}, here
we force ourselves to use only integrals of rough $1$-forms; the latter are reducible to Young integrals. In this way the technique~\cite{stepanov_sewing_17} might
remain ``hidden''; to avoid this we always give comments of how to formulate
the respective assertions with the help of integration of $2$-forms whenever necessary, so that the reader acquainted with the technique from~\cite{stepanov_sewing_17} might see it at work here.

\subsubsection*{Notation}

We use extensively the theory of Young integration and the notation from the recent expositions of rough paths theory~\cite{gubinelli_controlling_2004, friz_course_2014}. Here, we introduce some basic notation used throughout the paper, referring to Appendix \ref{appendix-notation} for more details  as well as proof of basic facts. 

 If $p$, $q \in \R^m$, $f: \R^m \to \R^k$, we write $f_p:= f(p)$ and $\delta f_{pq} := f_q - f_p$, $\delta_{pq} := q-p$.  We also  need the following notion of  integration of $1$-forms $f \d g $ along curves $\gamma: I \subseteq \R \to \Omega \subseteq \R^k$, provided that the regularity requirements for Young integration are satisfied. For example, if $f \in C^{\alpha}(\Omega; \R)$, $g \in C^{\beta}(\Omega; \R)$, with $\alpha+\beta>1$ and $\gamma \in C^1( I; \Omega)$, then $f \circ \gamma \in C^{\alpha}(I; \R)$, $g \circ \gamma \in C^{\beta}(I; \R)$ and so we define
\[ \int_{\gamma } f \d g := \int_I (f\circ \gamma) \, \d (g \circ \gamma).\]
As in the smooth case, the integral does not depend on the parametrization (except for the orientation). When $\gamma$ is a parametrization of an oriented segment from $p$ to $q$ (written $[pq] \subseteq \Omega$) we write $\int_{[pq]} f \d g :=  \int_{\gamma} f \d g$. In particular, this defines in a unique way the integral $\int_{\partial Q} f \d g$ of $f \d g$ on the boundary of an oriented rectangle $Q$ contained in $\Omega$. Precisely, we define an (oriented) rectangle $Q$ contained in $ \Omega$ (and write $Q\subseteq \Omega$) as an ordered triple of points $Q= [p; v_1, v_2]$, with $p$ standing for a ``base vertex'' and $v_1$,  $v_2$ for ``sides'' vectors, such that the convex envelope of  $\cur{p, p+v_1, p+v_2, p+v_1+v_2}$ is contained in $\Omega$. For simplicity, we also always require that $v_1$ and $v_2$ are parallel to vectors of the canonical basis of $\R^k$. The boundary of a rectangle $Q=[p; v_1, v_2]$ is defined as the formal sum of the four oriented segments $[p (p+v_1)]$, $[(p+v_1)(p+v_1+v_2)]$,  $[(p+v_1+v_2)(p+v_2)]$, $[(p+v_2)p]$ and $\int_{\partial Q} f \d g$ is then the sum of the four corresponding integrals. With a slight abuse of notation, sometimes we do not distinguish between the oriented rectangle $Q$ and the convex envelope of the four vertices, and write e.g., $q \in Q$ for a point in the convex envelope, or $\diam(Q)$ for its diameter.

\section{Differentiability with respect to a map}

\subsection{Derivatives with respect to a map}

We can rewrite problem~\eqref{eq:pfaff} in general space dimensions, i.e.\ for $\theta\colon I^m\to \R^d$, $g\colon I^m\to \R^k$ (in Section~\ref{sec_Intro1} we have the
 particular case $m=2$, $d=k=1$)  as in~\eqref{eq_pf_theta2}, i.e.,
\begin{equation}\label{eq:pfaff-germ} \delta \theta_{ p q } = f_{\bar{\theta}_p}   \delta g_{p q}  + o \bra{ \delta_{p q}}, \quad \text{for $p$, $q \in I^m$,}
\end{equation}
with $\bar{\theta}_p = (p, \theta_p)$ and $f_{\bar{\theta}_p}$ a $d\times k$ matrix. This suggests to restrict our investigation to maps $\theta$ which locally ``look like''  $g$, based on the validity of a Taylor expansion as in~\eqref{eq:pfaff-germ}. We formalize this notion in the following definition.

\begin{definition}[$g$-differentiable maps]
Let $g \in C(I^m; \R^k)$. We say that $\theta\colon I^m \to \R^d$ is 
(continuously) 
differentiable with respect to $g$, in brief, 
$g$-differentiable, if there exists a matrix-valued function
\[ \diff_g \theta = \bra{\partial_{g^i} \theta^j}_{i=1, \ldots, k}^{j=1, \ldots, d} \in C(I^m; \R^{d \times k}),\]
called $g$-derivative of $\theta$, such that
\begin{equation}\label{eq:g-differentiable} \delta \theta_{pq} =  (\diff_g \theta)_p  \delta g_{pq} + o( \delta_{pq} ) \quad \text{for $p$, $q \in I^m$.}\end{equation}
\end{definition}

\begin{remark}\label{rem_gdiff1}
The following easy assertions are valid.
\begin{itemize}
\item[(i)] A map $\theta\colon I^m \to \R^d$ is $C^1$, if and only if it is
$g$-differentiable with respect to $g(x):=x$, $g\colon I^m\to\R^m$ (in this case $k=m$) with $\diff_g \theta$ the usual differential (Jacobian) matrix of $\theta$.
\item[(ii)] Every $\theta \in C(I^m; \R^d)$ is clearly differentiable with respect to itself, i.e., $\theta$-differentiable, with $\diff_\theta \theta$ the identity $d\times d$ matrix (in this case $k=d$).  
\item[(iii)]  Uniqueness of the $g$-derivative is not true in general: consider for example a $C^1$ function $\theta\colon I\to \R$ and let $g(x) = (\theta(x), x)$ (seen a column), so that both 
the row matrix $(1,0)$ and the row matrix $(0,\theta'(x))$ with $\theta'$ the derivative of $\theta$,  
provide different $g$-derivatives of $\theta$. 
\item[(iv)] If for every $p\in I^m$ the differential of a scalar-valued $g$ in $p$ is nonzero
or does not exist
then the $g$-derivative of the map $\theta \in C(I^m; \R^d)$, if exists, is unique. In fact, if there are two, calling $v$ their 
difference, we have from~\eqref{eq:pfaff-germ} the relationship $v_p \delta g_{pq}= o(\delta_{pq})$ which,
minding that 
\[
\lim_{q_k\to p}  \frac{\delta g_{pq_k}}{\delta_{pq_k}}\neq 0
\]
for some sequence $\{q_k\}\subset I^n$, 
 is only possible when $v_p=0$.
\item[(vi)] For curves ($m=1$), the notion of $g$-differentiability 
can be compared to that of ``controlled''  paths first introduced  in~\cite{gubinelli_controlling_2004}. However, in that case, the definition 
of a controlled path $\theta\colon I\to \R^d$ requires e.g., $O(|\delta_{pq}|^{2 \alpha})$ instead of $o(\delta_{pq})$, when 
$g \in C^{\alpha}(I; \R)$ and $\alpha > 1/3$.
\end{itemize}
\end{remark}

\subsection{First order jets}

We may ask ourselves when for a given matrix-valued map $v$ and a vector valued map $g$ there is a $\theta$ such that $v = \diff_g \theta$. In the classical case $d=1$, $k=m$, $v\colon I^m\to \R^m$ smooth, and $g\colon I^m\to \R^m$ the identity map, this means that $v$ is a conservative vector field, and $\theta\colon I^m\to \R$ should be its potential (i.e.\ $v=\nabla \theta$, the classical gradient of $\theta$). The answer in this 
case can be given, for instance, in the following terms: the integral 
$\int_\gamma v\, \d x$ of the vector field $v$ (or, equivalently, of the differential form $v\,\d x$) 
along every closed smooth curve $\gamma$ in $I^m$ should vanish (of course, instead of testing the integrals over all closed curves, it is enough to take them over rectangles in $I^m$ with sides parallel to the coordinate axes).

For the general case of not necessarily smooth maps $v$ and $g$ such an answer is only possible when a suitable notion of the integral of the ``vector valued differential form'' $v\,\d g$ over closed curves is defined. This is the case, for instance, of H\"{o}lder maps $v \in C^{\alpha}(I^m; \R^{d\times k})$, $g \in C^{\beta}(I^m; \R^k)$ with $\alpha+\beta>1$, the respective integral being intended in the sense of Young. In particular, the 
following theorem, proven in Appendix~\ref{appendix-jets}, is valid.

\begin{theorem}[integration of $g$-jets]\label{thm:jets}
Let $0 \in I$ and let $\alpha$, $\beta \in (0,1]$, with $\alpha+\beta>1$, $g \in C^{\beta}(I^m; \R^k)$, $v \in C^{\alpha}(I^m; \R^{d\times k})$. 
Then, there exists a $\theta \in C(I^m; \R^d)$ such that $v = \diff_g \theta$, if and only if 
\begin{equation}\label{eq:g-jet}
\int_{\partial Q} v \d g = 0 
\end{equation}
for every rectangle $Q \subseteq I^m$ with sides parallel to the coordinate axes,
or, equivalently, 
\begin{equation}\label{eq:g-jet1gam}
\int_{\gamma} v \d g := \int_0^1  v(\gamma(t)) \d g(\gamma(t)) =0 
\end{equation}
for all closed Lispchitz curves $\gamma\colon [0,1]\to I^m$. 
In this case we call $v$ an \textit{$\alpha$-H\"older $g$-jet}, and write 
$v \in \jet^{\alpha}_g(I^m; \R^{d \times k})$. In such a case, the identity 
\begin{equation}\label{eq:integration-g-diff} \delta \theta_{pq} = \int_{[pq]} v\d g = \int_{[pq]} ( \diff_g \theta) \d g \quad \text{for $p$, $q \in I^m$}\end{equation}
holds, hence $\theta$ is uniquely determined up to an additive constant, and 
denoting by $\theta$ 
the  unique map such that~\eqref{eq:integration-g-diff} holds and $\theta_0 = 0$, we have that the operator
\[ v \in \jet^{\alpha}_g (I^m; \R^{d \times k}) \mapsto \theta \in C^{\beta}(I^m; \R^d)\]
is well-defined, linear and continuous 
(the space $\jet^{\alpha}_g (I^m; \R^{d \times k})$ being equipped with the component-wise H\"{o}lder norm
$\|\cdot\|_{C^\alpha}$). 
\end{theorem}

\begin{remark}
We notice that, under conditions of Theorem~\ref{thm:jets},
if~\eqref{eq:integration-g-diff} holds, then 
by the fundamental estimate of Young integrals~\eqref{eq:young-curves-germs} we obtain
\begin{equation}\label{eq:integration-g-jet-2} |\delta \theta_{pq} - (\diff_g \theta)_p \cdot \delta g_{pq} | \le  \c [\delta g]_{\beta} [\delta (\diff_g \theta)]_{\alpha} | \delta_{pq} |^{\alpha+\beta} = O( |\delta_{pq}|^{\alpha+\beta}) \quad \text{for $p$, $q \in I^m$,}\end{equation}
for $\c = \c(m, \alpha, \beta)$, hence improving the $o(\delta_{pq})$ term in~\eqref{eq:g-differentiable}. In particular, this yields that \eqref{eq:g-jet1gam} also holds for $\gamma \in C^{\sigma}([0,1]; I^m)$, with  $\sigma(\alpha+\beta)>1$, by  approximation of H\"older curves with Lipschitz ones and continuity of Young integral.
\end{remark}

\subsection{Chain rule for derivatives}

As a simple consequence of~\eqref{eq:integration-g-jet-2}, we get the following
chain rule (a precise and more general statement of it is provided by Proposition~\ref{prop:chain-rule}).

\begin{corollary}[chain rule]
If $f$ is $h$-differentiable and $h \circ \theta$ is $g$-differentiable, then $f\circ \theta$ is $g$-differentiable, with 
\begin{equation}\label{eq_chain_rule_isotr1}
\diff_g (f \circ \theta) = (\diff_h f)_{\theta} \diff_g (h\circ \theta),
\end{equation}
provided that  $\diff_h f \in C^\alpha$, $h \in C^\beta$ and $\theta \in C^\sigma$, for $\alpha$, $\beta$, $\sigma \in (0,1]$, with $\sigma(\alpha+\beta) >1$. 
\end{corollary}

\begin{proof}
The estimate~\eqref{eq:integration-g-jet-2} (with $f$ and $h$ in place of $\theta$ and $g$ respectively)
implies
\[
 \delta f_{pq} =  (\diff_h f)_p  \delta h_{pq} + O( |\delta_{pq}|^{\alpha+\beta} ), 
\]
for $p,q$ in the domain of definition of $f$,
and hence
\begin{align*}
 (\delta f\circ\theta)_{st} & =(\diff_h f)_{\theta_s}  \cdot \delta h_{\theta_s\theta_t} + 
O\left( |\delta_{\theta_s\theta_t}|^{\alpha+\beta} \right)\\
& = (\diff_h f)\circ\theta_{s}  \left(\delta (h\circ\theta)\right)_{st} + O\left( |\delta_{st}|^{\sigma(\alpha+\beta)} \right)\\
&=  ((\diff_h f)\circ\theta)_{s}  \left(\diff_g (h\circ\theta)_{s} \delta g_{st} + o(\delta_{st}) \right)
+O\left( |\delta_{st}|^{\sigma(\alpha+\beta)} \right)\\
& = ((\diff_h f)\circ\theta)_{p}  \diff_g (h\circ\theta)_{p} \delta g_{st}+ o(\delta_{st}),
\end{align*}
for $t,s$ in the domain of definition of $\theta$, as claimed.
\end{proof}

Let us consider the following examples. 

\begin{example}[composition]\label{ex:chain-rule}
Let $f\colon J^n \to \R^d$ be differentiable in the classical sense, 
$D f = \diff_{\mbox{id}}f \in C^{\alpha}(J^n; \R^{d \times n})$ and let $\theta = g \in C^{\beta}(I^m; \R^n)$,
with $\beta(1+\alpha)>1$, 
we deduce from~\eqref{eq_chain_rule_isotr1} with $h:=\mbox{id}$ using 
$\mbox{id}\circ \theta = g$ and Remark~\ref{rem_gdiff1}
that $f\circ g$ is $g$-differentiable with $\diff_g (f \circ g) = (D f) \circ g$.
\end{example}

\begin{example}[composition with graphs]\label{ex:anisotropic}
This is in fact an application of Proposition~\ref{prop:chain-rule}, but we state it here for later use. Given  $g \in C^{\beta}(I^m; \R^k)$, define
\[ h_{(x^m, x^n) } :=  (g_{x^m}, x^n), \quad \text{ for $x =(x^m, x^n) \in I^{m +n}$,}\]
 let $f\colon I^{m+n} \to \R^d$ be $h$-differentiable, with $\diff_h f = (\diff_g f, \diff_{x^n} f) \in C^{\alpha}( I^{m+n} ; \R^{d \times( k + n)})$, and  $\theta \in C^{\beta}(I^m; \R^n)$ be $g$-differentiable. Letting $p \mapsto \bar{\theta}_p := (p, \theta_p)$ be its graph, we have that $h \circ\bar{ \theta} = (g, \theta)$ is also $g$-differentiable, with $\diff_g ( h \circ\bar{ \theta}) = (\Id, \diff_g \theta)$, hence we conclude by Proposition~\ref{prop:chain-rule} that $f_{ \bar{\theta}}$ is $g$-differentiable, with
\begin{equation}\label{eq:chain-rule-graph} \diff_g f_{ \bar{\theta}} =(\diff_{g} f)_{\bar{\theta}} + (\diff_{x^n} f)_{\bar{\theta}} \diff_g \theta.\end{equation}
\end{example}

\subsection{Jets and $g$-differentiable maps}

Let now $g = (g^i)_{i=1}^k\colon I^m\to \R^k$. We are interested in knowing when  a given $g$-differentiable map
$v = (v^i)_{i=1}^k \colon I^m\to \R^k$, 
is a $g$-derivative of
some real valued function $\theta \in C(I^m)$. 
 If  $v$ and $g$ are H\"{o}lder continuous, a sufficient condition for the positive answer is given 
by the following result. 

\begin{proposition}\label{prop:zust-1}
Let $\alpha$, $\beta \in (0,1]$, with $\alpha+2\beta>2$, $g =(g^i)_{i=1}^k\in C^{\beta}(I^m; \R^k)$. If $v  = (v^i)_{i=1}^k \in C^\beta(I^m; \R^k)$ is $g$-differentiable, with $\diff_g v \in C^{\alpha}(I^m; \R^{k\times k})$ and for every $i  \neq j$ either 
\begin{equation}\label{eq_frob1cond1}
\partial_{g^i} v^j = \partial_{g^j} v^i,
\end{equation}
 or for every rectangle $Q \subseteq I^m$ with sides parallel to coordinate axes one has
\begin{equation}\label{eq:wedge-null} 
\int_{\partial Q} g^i \d g^j = 0, 
\end{equation} 
then $v \in \jet^{\beta}_g(I^m; \R^k)$, so that in particular, by 
Theorem~\ref{thm:jets}, $v=\diff_g\theta$ for some $\theta\in C^{\beta}(I^m)$.
\end{proposition}

Note that with the Stokes' theorem from~\cite{stepanov_sewing_17} at hand we could reformulate~\eqref{eq:wedge-null} as
\[
\int_{Q} \d g^i \wedge \d g^j = 0
\]
for every rectange $Q$, i.e.\ just symbolically
\[
\d g^i \wedge \d g^j = 0.
\]
We postpone the proof of Proposition~\ref{prop:zust-1} to Appendix~\ref{appendix-jets}, 
but will give its rough (though only formal) idea here.

\begin{proof}[Idea of the proof of Proposition~\ref{prop:zust-1}:]
In view of Theorem~\ref{thm:jets} it is enough to show that $v$ is
a $g$-jet, that is, when $\int_{\partial Q} v \d g = 0$
for every rectangle $Q \subseteq I^m$ with sides parallel to the coordinate axes.
A purely formal application of calculus rules yields the identities
\[  
\int_{\partial Q} v \cdot \d g = \int_Q \d v \wedge \d g = \sum_{i < j} \int_Q  (\partial_{g^i} v^j - \partial_{g^j} v^i) \d g^i \wedge \d g^j,
\]
where we used (formally) the Stokes' theorem, the $g$-differentiability of $v$ 
and the antisymmetry of ``differential $2$-forms'' 
$\d g^i \wedge \d g^j = - \d g ^j \wedge \d g^i$ 
(yielding in particular $\d g^i \wedge \d g^i= 0$). 
Thus, a sufficient condition for the left-hand side integral to vanish is that, for every $i\neq j$ either~\eqref{eq_frob1cond1} holds, or
$\d g^i \wedge \d g^j = 0$, 
the  latter condition being equivalent to~\eqref{eq:wedge-null} (this formulation
avoids two-dimensional integrals which we did not introduce here, although one could have 
used the integrals introduced by R.\ Zust in~\cite{zust_integration_2011} or equivalently those in \cite{stepanov_sewing_17}). Note that 
the requirement $\alpha+2\beta>2$ may be seen as necessary for all the  two-dimensional integrals above in fact exist in the sense of R.\ Zust. 
\end{proof}


%
%

Note that in the classical case $k=m$ and $g$ identity, 
$\d g^i\wedge \d g^j=\d x^i\wedge \d x^j\neq 0$ when $i\neq j$, so that~\eqref{eq:wedge-null}
is never valid, and Proposition~\ref{prop:zust-1} provides 
just the condition~\eqref{eq_frob1cond1} for $v$ to be a gradient. The latter condition
becomes just $\partial_{x^i} v^j = \partial_{x^j} v^i$ for all $i\neq j$, that is,
the classical curl-free condition $\nabla\times v=0$. In other words, in the general situation~\eqref{eq_frob1cond1} can be seen as just a natural extension of curl-free condition. On the contrary,~\eqref{eq:wedge-null} is essentially new and absent in the classical case. The following lemma provides examples of $g$ such that~\eqref{eq:wedge-null} holds.

\begin{lemma}\label{lem:examples-dg-dg-null}
Let $\alpha$, $\beta$, $\sigma \in (0,1]$, with $\alpha+2 \beta>2$ and $2 \beta \sigma >1$, $w \in C^{\beta}(I^m; \R)$, $h \in C^{\beta}(I^m; \R^n)$ be $w$-differentiable, with $\diff_w h \in C^{\alpha}(I^m; \R^n)$, and $f \in \C^{\sigma}(\R^n; \R^k)$. Then, $g := f \circ h \in C^{\beta \sigma}(I^m; \R^k)$ satisfies~\eqref{eq:wedge-null}. 
\end{lemma}

The proof of the above Lemma~\ref{lem:examples-dg-dg-null} provided in Appendix~\ref{appendix-jets}, essentially relies on the fact that, for a real-valued $w$, 
one (formally) has $\d w \wedge \d w = \d (w \d w) = \frac 1 2 \d (\d w^2) = 0$, hence if all the components $(g^i)_{i=1}^k$  locally ``look like'' $w$, one may conclude that~\eqref{eq:wedge-null} holds.

\begin{remark}\label{rm_zust_tree1}
It is not clear whether all maps $g$ satisfying \eqref{eq:wedge-null} have necessarily to be in the form $g = f\circ h$ as in the above lemma. A  weaker result of this ilk however holds due to theorems~1.2 and~1.1 from~\cite{zust_2015}. Namely, if
$m=k=2$, i.e. $g\colon \R^2\to\R^2$ satisfies~\eqref{eq:wedge-null} then there is a metric tree $T$, i.e.\ a complete metric space such that every couple of its points is connected by a unique arc (a continuous injective image of the unit interval),
and surjective map $f \colon \R^2\to T$ and $h \colon T\to g(\R^2)\subset \R^2$ such that $g=f \circ h$. Since $T$ is a one-dimensional space in a natural sense, then it means, informally, that the components $g^1$ and $g^2$ depend on ``a single coordinate'' (i.e.\ the point on $T$) that we may regard as an analogue of $w$ in Lemma~\ref{lem:examples-dg-dg-null}.
\end{remark}

At last we remark that even when $v$ is not a $g$-jet, in some cases we can ``correct it'' 
by an adding a ``corrector'' map 
to obtain a $g$-jet as the following proposition shows.


\begin{proposition}[``correctors'' and jets]\label{prop:compensated1}
Let $0 \in I$ and let $\alpha$, $\beta \in (0,1]$, with $\alpha+2\beta>2$, $g \in C^{\beta}(I^2; \R^2)$ with $g_{(s,t)} =( g^1_s, g^2_{(s,t)} )$ for $(s,t) \in I^2$. Let $v  = (v^1, v^2) \in C^{\beta}(I^2; \R^2)$ be $g$-differentiable with $\diff_g v \in C^{\alpha}(I^2; \R^2)$ and define, for $(s,t) \in I^2$,
\[ \mathfrak{v}_{(s,t)} := \int_{[(s, 0) (s, t) ] } \bra{ \partial _{g^1} v^2 -  \partial_{g^2} v^1}  \d g^2.\]
Then $(v^1 - \mathfrak{v}, v^2) \in \jet^{\alpha+\beta-1}_g(I^2)$.
More precisely, $\mathfrak{v}$ has the following anisotropic regularity:
\[ \mathfrak{v}_{(s,\cdot)} \in  C^{\beta}(I; \R), \quad \text{and} \quad \mathfrak{v}_{(\cdot,t)} \in  C^{\alpha+\beta-1}(I; \R) \quad \text{for $s$, $t \in I$.}\]
\end{proposition}

Again we postpone the proof of Proposition~\ref{prop:compensated1} to Appendix~\ref{sec_proof_compensated1}, 
providing here only its rough and purely formal idea.

\begin{proof}[Formal idea of the proof of Proposition~\ref{prop:compensated1}:]
For  $Q = [s^0, s^1] \times[t^0, t^1] \subseteq I^2$ we calculate
\[\begin{split}
 \int_{\partial Q} v \cdot \d g & = \int_{ Q} v^1   \d g^1 + v^2 \d g^2\\ & = \int_{\partial Q} 1 \d v^1 \wedge \d g^1 + 1 \d v^2  \wedge \d g^2\quad \text{by Stokes' theorem}\\
 & = \int_{Q} (-\partial_{g^2} v^1 + \partial_{g^1} v^2) \d g^1 \wedge \d g^2  \quad \text{since $v$ is $g$-differentiable} \\
  & = \int_{s^0}^{s^1}\d g^1_s  \int_{t^0}^{t^1} (-\partial_{g^2} v^1 + \partial_{g^1} v^2)_{(s,t)} \d g^2_{(s, \cdot)} (t)    \quad \text{by Fubini theorem}\\
 & = \int_{s^0}^{s^1} \bra{ \mathfrak{v}_{(s,t^1)}  - \mathfrak{v}_{(s,t^0)} } \d g^1(s) = \int _{\partial Q} \mathfrak{v} \d g^1. \qedhere
 \end{split}\]
\end{proof}

\section{An implicit function theorem}


In this section, we show that a version of the implicit function theorem holds, for $h$-differentiable maps, provided that $h$ has a suitable ``product'' structure: we assume indeed that
\begin{equation}\label{eq:implicit-h} h\colon I^{m+n} \to \R^{k+n}, \quad h_x := (g_{x^m}, x^n), \quad \text{for $x = (x^m, x^n) \in I^{m+n}$,}\end{equation}
for some $g \in  C^{\beta}(I^{m}; \R^k)$. Given an $h$-differentiable map $f\colon I^{m+n} \to  \R^n$, writing $\diff_h f = (\diff_g f, \diff_{x^n} f)$, we say that a point $x_0  \in I^{m+n}$ is \emph{non-degenerate} if $(\diff_{x^n} f)_{x_0}$ is invertible. The structure of level sets $f^{-1}(f_{x_0})$ at non-degenerate points is described in the following result.

\begin{theorem}[implicit function theorem]\label{thm:dini}
Let $\beta, \gamma \in (0,1]$ with $\beta(1+ \gamma) >1$, $g \in C^{\beta}(I^{m} ; \R^k)$, let $h$ be as in~\eqref{eq:implicit-h}, $f\colon I^{m+n}  \to \R^n$ be $h$-differentiable, with $\diff_h f \in C^{\gamma}( I^{m+n} ; \R^{n \times(k+ n)})$ and $x_0 \in I^{m+n}$ be non-degenerate. Then, there exist open sets $J^m$, $J^n$ such that $x_0 \in J^m \times J^n\subseteq I^{m+n}$ and a unique $\theta\colon J^m \to  I^n$ such that
\[ f^{-1}(f_{x_0})\cap J^{m}\times J^n = \bar{\theta}(J^m),\]
where $\bar{\theta}_p := (p, \theta_p)$.
Moreover, $\theta$ is $g$-differentiable, with 
\begin{equation}\label{eq:g-dini} \diff_g \theta = - (\diff_{x^n} f)_{\bar{\theta}}^{-1} (\diff_g f)_{\bar{\theta}} \in C^{\beta \gamma}(J^m; \R^{n \times k}). \end{equation}
\end{theorem}

A detailed proof of this result is provided in Appendix~\ref{appendix-dini}. As in the classical implicit function theorem, we rely on a fixed point argument, using the improved error estimates~\eqref{eq:integration-g-jet-2}. Let us point out that the fixed point map is not directly induced by~\eqref{eq:g-dini}, instead it is built by choosing suitable  differences between Taylor expansion at the point $x_0$. The validity of~\eqref{eq:g-dini}, in particular the fact that the right hand side therein is a $g$-jet, follows a posteriori.

As an application of Theorem~\ref{thm:dini}, we provide a partial answer to the following natural question, converse to the chain rule (Example~\ref{ex:chain-rule}): when $g$-differentiable functions $f$ are necessarily obtained via pointwise composition $f = F \circ g$? 

\begin{example}When, $m=1$, the example of a Young integral
\[ f_t = \int_{0}^t h_s \d g_s, \quad \text{for $t \in I$,}\]
shows that this is not always the case, since $f$ depends non-locally 
on $g$, e.g., by modifying $g$ only around $0$, the values of $f_t$ may change also for large  $t$'s. Still, the fundamental estimate of Young integrals \eqref{eq:young-germ} gives that $f$ is $g$-differentiable with $\diff_g f = h$.
\end{example}

To provide a sufficient condition ensuring that it must be $f = F \circ g$ for $g$-differentiable $f$'s, it is sufficient to argue that $f$ is constant on any level set $\cur{g = c}$. From~\eqref{eq:g-differentiable}, letting $p$, $q \in \cur{g = c}$, we would deduce that
\[ \delta f_{pq} = (\diff_g f)_p \cdot  \delta g_{pq} + o(\delta_{pq})  = o ( \delta_{pq}).\]
Therefore, if $p$ and $q$ can be connected via a continuous curve $\theta$ with values in $\cur{g= c}$, one has
\[(\delta f)_{ \theta_s \theta_t} = o ( \delta \theta_{st}), \]
which leads (Lemma~\ref{lem:dyadic}) to $\delta f_{\theta} = 0$, provided that $o( \delta \theta_{st}) = o(\delta_{st})$. Hence, such argument holds whenever one has a sufficiently precise parametrization of the level sets of $g$: in the next proposition, whose proof can be found in Appendix~\ref{appendix-dini},  we give a precise statement (notice that we slightly shift the notation from the discussion above, following instead that of Theorem~\ref{thm:dini}).

\begin{proposition}[$f$-differentiable maps induced by composition]\label{prop:g-induced-composition}
Let $\beta$, $\gamma$, $g$, $h$, $f$ and $x_0$ be as in Theorem~\ref{thm:dini}. If $\varphi\colon I^{m+n} \to \R^k$ is $f$-differentiable, with $\diff_f \varphi \in C^{\gamma}(I^{m+n}; \R^{k \times n})$, then there exists $J^m$, $J^n$ with $x_0 \in J^m \times J^n \subseteq I^{m+n}$  and $\Phi\colon \R^n \to \R^k$ such that
\[ \varphi = \Phi \circ f \quad \text{on $J^{m} \times J^n$.}\]
\end{proposition}

\section{Frobenius theorems}

Given $g\colon I^m \to \R^k$ and $f\colon I^m \times \R^k \to \R^{d \times k}$, 
 we provide sufficient conditions ensuring existence and uniqueness of  $\theta\colon I^m \to \R^d$ such that~\eqref{eq:pfaff-germ} holds 
with $\theta_{p_0} = \vartheta$, for given $p_0 \in I^m$, $\vartheta \in \R^d$.

\begin{remark}[rough exterior differential systems]
The relationship \eqref{eq:pfaff-germ}, seen as a problem for $\theta$, can be intepreted as an exterior differential system, by defining the following ``rough differential forms'' on $I^m \times \R^d$,
\[ \omega^j := \d x^{m+i}- \sum_{\ell = 1}^k f^j_{\ell} \d g^\ell \quad \text{for $j \in \cur{1, \ldots, d}$,}\]
and finding its H\"older integral manifold $\bar{\theta}: I^m \to I^m \times \R^d$, i.e., such that  $\bar{\theta}_{p_0} = (p_0, \vartheta)$ and
\begin{eqnarray}
\label{eq:eds-1}\bar\theta^*\,\d x^i &= \d s^i,  \quad \text{for $i \in \cur{1, \ldots, m}$,}\\
\label{eq:eds-2}\bar\theta^*\,\omega^j &=0,  \quad \text{for $j \in \cur{1, \ldots, d}$.}
\end{eqnarray}
The solution $\bar{\theta}$ to \eqref{eq:eds-1} \eqref{eq:eds-2} will be given by $\bar{\theta}_{s}= (p_0 + s, \theta_s)$ where $\theta$ satisfies \eqref{eq:pfaff-germ}.
\end{remark}

 As anticipated in the introduction, we can equivalently restate~\eqref{eq:pfaff-germ} as a (Young) integral equation, which is here justified as a consequence of Theorem~\ref{thm:jets}.

\begin{lemma}\label{lem:equivalence-pfaff}
Let $\beta$, $\gamma \in (0,1]$ with $\beta(1+\gamma)>1$, let \[ g \in C^{\beta}(I^m; \R^k), \quad f \in C^{\gamma}(I^m \times \R^k; \R^{d \times k}), \quad \theta \in C(I^m; \R^d),\]
with $\theta_{p_0} = \vartheta \in I^m$. Then,~\eqref{eq:pfaff-germ} with $\bar{\theta}_p :=(p, \theta_p)$ holds if and only either of the  following conditions hold:
\begin{enumerate}
\item the identity $\delta \theta_{pq} =  \displaystyle\int_{[pq]} f_{\bar{\theta}} \cdot \d g$ holds, for every $p$, $q \in I^m$,
\item $\theta \in C^{\beta}(I^m; \R^k)$ is $g$-differentiable with $\diff_g \theta = f_{\bar{\theta}} \in C^{\beta \gamma}(I^m; \R^{d \times k})$,
\item there exists a $v \in \jet^{\beta \gamma}(I^m; \R^{d \times k})$ such that $\theta_p = \vartheta + \int_{[p_0p]} v \cdot \d g$ and
\[ f_{\bar{\theta}_p} = v_p \quad \text{for every $p \in I^m$.}\]
\end{enumerate}
\end{lemma}

\begin{proof}
The proof is straightforward: the validity of \eqref{eq:pfaff-germ} is indeed what inspired the general definition of $g$-differentiability, so that condition \emph{(2)} is plainly equivalent to it. Young integration along segments and the fundamental estimate \eqref{eq:young-curves-germs} shows the equivalence with condition \emph{(1)}, which easily implies \emph{(3)}. In turn, \emph{(3)} implies $g$-differentiability, i.e.\ condition \emph{(2)} by Theorem~\ref{thm:jets}.
\end{proof}

%
Each formulation suggests a fixed point formulation: \emph{(1)} and \emph{(2)} in the space of maps $\theta$, \emph{(3)} in the space of $g$-jets, defining a map $F(v)$ via
\begin{equation}\label{eq:map-F-frobenius} v \mapsto \theta := \vartheta + \int_{[p_0 \cdot ]} v \d g \mapsto F(v) := f_{\bar{\theta}}.\end{equation}
From such a point of view, we see that there are obstructions for $F(v)$ being a $g$-jet at least as regular as $v$, both of algebraic and analytical nature (see Proposition~\ref{prop:zust-1}), and we solve them by assuming suitable ``involutivity'' assumptions. In our first Frobenius-type result, we assume the validity of $\d g^i \wedge \d g^j = 0$ in the sense of~\eqref{eq:wedge-null}. The case of a scalar valued $g\colon I^m \to \R$, as in the Introduction, is covered also by the statement.

\begin{theorem}[case $\d g^i \wedge \d g^j = 0$]\label{thm:frob-1}
Let $\beta$, $\gamma \in (0,1]$ with $\beta(2+\gamma)>2$,  $g\in C^{\beta}(I^m; \R^k)$ be such that~\eqref{eq:wedge-null} holds for every $i \neq j$,  and $f\colon I^m \times \R^d \to \R^{d\times k}$ be $(g, x^d)$-differentiable, 
i.e.\ differentiable with respect to the map
\[
(x^m,x^d)\in I^m\times \R^d\mapsto (g(x^m), x^d)\in \R^{k+d},
\]
and
\[ \diff_{(g, x^d)} f =(\diff_{g} f, \diff_{x^d} f) \in 
{C^{\gamma}(I^m \times \R^{d}; \R^{(dk)\times(k+d)}}). \]
Then, for every $(p_0, \vartheta) \in I^m \times \R^d$, there exists a unique $\theta \in C^{\beta}(I^m; \R^d)$ such that ~\eqref{eq:pfaff-germ} holds and $\theta_{p_0} = \vartheta$.
\end{theorem}

The proof follows a Banach fixed point argument in the subspace of jets $v \in \jet^{\beta\gamma}(I^m; \R^{d \times k})$ such that $v_{p_0} = f(p_0, \vartheta)$.  We define the map $F$ via~\eqref{eq:map-F-frobenius}, that is,
\[ F( v )_p := f \bra{p, \vartheta + \textstyle{ \int_{[p_0 p]} v \d g }  } \quad \text{for $p \in I^m$.}\]
Indeed, if $v \in \jet^{\beta\gamma}_g(I^m; \R^{d \times k})$ is a fixed point for $F$, then the series of identities
\begin{equation}\label{eq:fixed-point-series-identities} \theta_p := \vartheta+ \int_{[p_0 p]} v \d g = \vartheta + \int_{[p_0 p]} F(v) \d g = \vartheta + \int_{[p_0 p]} f(q, \theta_q) \d g_q\end{equation}
and Lemma~\ref{lem:equivalence-pfaff} give that $\theta$ defined by~\eqref{eq:fixed-point-series-identities} is a solution to~\eqref{eq:pfaff-germ} with $\theta_{p_0} = \vartheta$. 
The only non-trivial part is showing that $F$ maps jets into jets, which follows from Proposition~\ref{prop:zust-1}; details are provided in Appendix~\ref{appendix-frobenius}. 

In our second Frobenius-type result, we assume instead that $k=m$ and $g = (g^i)_{i=1}^m$ admits a ``diagonal'' structure, i.e., for every $i \in \cur{1, \ldots, k}$, the map $s \mapsto g^i_s$ is actually a function of the single coordinate $s^i$.  Notice that this assumption covers the classical case $g$ being the identity map.

\begin{theorem}[``diagonal'' case]\label{thm:frob-2}
Let $\beta, \gamma \in (0,1]$ with $\beta(2+\gamma) > 2$,  $g \in C^{\beta}(I^m; \R^m)$ be in the form
\[ g_s = (g^1_{s^1}, \ldots, g^m_{s^m})\quad \text{for $s = (s^1, s^2, \ldots, s^m) \in I^m$.}\]
and $f \colon I^m \times \R^d \to \R^{d \times m}$ be $(g, x^d)$-differentiable with

\[  \diff_{(g, x^d)} f  = (\diff_{g} f, \diff_{x^d} f) \in C^{\gamma}(I^m \times \R; \R^{(dm) \times (m+d)}),\]
and such that the following ``involutivity'' conditions hold in $I^m \times \R^d$
\begin{equation}\label{eq:involutivity} f^{\ell,i} \partial_{x^{\ell'}} f^{\ell, j} - f^{\ell, j} \partial_{x^{\ell'}}f^{\ell,i} = 0 \quad  \text{and} \quad  \partial_{g^i} f^{\ell,j}  - \partial_{g^j}f^{\ell, i} = 0, \quad \text{for $i$, $j \in \cur{1, \ldots, m}$,  $i \neq j$,} \end{equation}
and $\ell, \ell' \in \cur{1, \ldots, d}$.
Then, for every $(p_0, \vartheta) \in I^m \times \R^d$ there exists a unique $\theta \in C^{\beta}(I^m; \R^d)$ such that $\theta_{p_0} = \vartheta$ and~\eqref{eq:pfaff-germ} holds.
\end{theorem}

\begin{remark}[``block'' diagonal structure] We may extend the result above in the more general case  $k = \sum_{i=1}^m k_i$, $g \in C^{\beta}(I^m; \R^k)$ in the form
\begin{equation}\label{eq:blocks} g_s = (g^1_{s^1}, \ldots, g^{k_1}_{s^1}, g^{k_1+1}_{s^2}, \ldots, g^{k_1+k_2}_{s^2}, \ldots, g^{k}_{s^m})\quad \text{for $s = (s^1, s^2, \ldots, s^m) \in I^m$.}\end{equation}
In this case, $f : I^m \times \R^d \to \R^{d \times k}$ and the involutivity condition \eqref{eq:involutivity} reads
\[ f^{\ell,i} \partial_{x^{\ell'}} f^{\ell, j} - f^{\ell, j} \partial_{x^{\ell'}}f^{\ell,i} = 0 \quad  \text{and} \quad  \partial_{g^i} f^{\ell,j}  - \partial_{g^j}f^{\ell, i} = 0, \quad \text{for $i$, $j \in \cur{1, \ldots, k}$,} \]
and $\ell, \ell' \in \cur{1, \ldots, d}$ are such that the functions $g^i$, $g^j$ depend on different variables, i.e., belong to different blocks of \eqref{eq:blocks}. For simplicity we provide a proof only in the easier case of blocks of size $1$, as in the stated theorem.
\end{remark}

\begin{proof}[Idea of the proof of Theorem~\ref{thm:frob-2}:]
The proof of this result, the details of which are given in Appendix~\ref{appendix-frobenius}, is by induction over $\ell \in \cur{1, \ldots, m}$. For ease of notation, let us assume that $p_0 = 0$ and argue in the scalar-valued case, $d=1$. The case $\ell=1$ is then that of Young differential equations, for which the theory is well-established.  In the induction step from $\ell$ to $\ell+1$, we use the inductive assumption to define $\theta\colon I^{\ell+1} \to \R$ on the $\ell$-dimensional space $I^\ell \times \cur{0}$ such that the restriction of the system of equations \eqref{eq:pfaff-germ} holds, we extend it by solving the Young differential equation (with respect to variable $t \in I$)
\begin{equation}\label{eq:theta-equation} \theta_{(p, t)} = \theta_{(p,0)} + \int_{[(p,0)(p,t)]} f^{\ell+1}_{\bar{\theta}} \d g^{\ell+1}, \quad \text{for $t \in I$, $p \in I^\ell$}.\end{equation}
In view of the regularity assumptions on $f$, there exists a unique solution to 
\eqref{eq:theta-equation}, hence $\theta$ is well-defined and $t \mapsto \theta_{(p,t)}$ is H\"older (at every $p \in I^\ell$). The crucial point is to show that $\theta \in C^{\beta}(I^{\ell+1}; \R^d)$ and  solves~\eqref{eq:pfaff-germ}. To this aim, we use 
Theorem~\ref{thm:g-differentiable-yde},
that provides $g$-differentiability of $\theta$ with an explicit equation for $\partial_{g^i} \theta$, for $i \in \cur{1, \ldots, \ell}$, formally obtained by $g$-differentiation of~\eqref{eq:theta-equation}. Then, using the involutivity assumptions~\eqref{eq:involutivity}, we conclude that, for $i \in \cur{1, \ldots, \ell}$, the identity
\[ (\partial_{g^i} \theta)_{(p, t)} -f^i_{\bar{\theta}_{(p,t)}} = \int_{[(p,0)(p,t)]} (\diff_{x^d} f^{\ell+1})_{\bar{\theta}} \bra{ \partial_{g^i} \theta  - f^{i}_{\bar{\theta}} } \d g^{\ell+1} \quad \text{holds for $t \in I$, $p \in I^\ell$.}\]
Finally, a Gronwall-type argument (Lemma~\ref{lem:gronwall}) yields $\partial_{g^i}\theta = f^i_{\bar{\theta}}$, hence the thesis.
\end{proof}

\section{Examples and open questions}

We conjecture that both Theorem~\ref{thm:frob-1} and Theorem~\ref{thm:frob-2} are instances of a more general result, where one can drop any ``diagonal assumption'' on $g = (g^i)_{i=1}^k$ and simply assume that, given $f = (f^i)_{i=1}^k$, for any $i$, $j \in \cur{1, \ldots, k}$, at least one between $\d g^i \wedge \d g^j = 0$ in the form~\eqref{eq:wedge-null} or~\eqref{eq:involutivity} hold. Both strategies of proof, at present, do not allow us to conclude.  Indeed, in the proof of Theorem~\ref{thm:frob-1}, jet spaces are not stable with respect to the map $F$ ~\eqref{eq:map-F-frobenius}. In the following example, we show that one can try to introduce a new map $\mathfrak{F}$ sending jets into jets such that its fixed points are solutions to~\eqref{eq:pfaff}. However, this approach still faces difficulties due to loss of regularity. 

\begin{example}[triangular case]
In the setting of 
Proposition~\ref{prop:compensated1}
let 
\[ (v^1, v^2) \mapsto (f_{\bar{\theta}}^1 - \mathfrak{v},   f_{\bar{\theta}}^2 )\]
where $\theta_p := \vartheta + \int_{[p_0 p]} v \d g$ and 
$ \mathfrak{v}_{(s,t)} := \int_{[(s, 0) (s, t) ] }  \bra{ \partial _{g^1} f^2_{\bar{\theta}} -  \partial_{g^2} f^1_{\bar{\theta}}}  \d g^2$.
One has that any fixed point would provide a solution to~\eqref{eq:pfaff}. Indeed, (arguing formally) we have
\[ \begin{split} \mathfrak{v} & = \int_{[(s, 0) (s, t) ] } \bra{\partial _{g^1} v^1 f^2_{\bar{\theta}}  -  v^2 \partial_{g^2} f^1_{\bar{\theta}}} \d g^2 \quad \text{by the chain rule,}\\
 & = \int_{[(s, 0) (s, t) ] } \bra{  \partial _{g^1} (f^1_{\bar{\theta} } - \mathfrak{v})  f^2_{\bar{\theta}}  -  f^2 \partial_{g^2} f^1_{\bar{\theta}}} \d g^2 \quad \text{being $(v^1, v^2)$ a fixed point,}\\
 & =  - \int_{[(s, 0) (s, t) ] }  \mathfrak{v} \,  \partial_{g^1} f^2_{\bar{\theta}} \d g^2\quad \text{by involutivity~\eqref{eq:involutivity},} \\
 & = 0 \quad \text{by Lemma~\ref{lem:gronwall}.}
\end{split}\]
\end{example}

In the inductive proof of Theorem~\ref{thm:frob-2}, instead, it is not clear how to show that $\theta$, defined by solving $\d \theta = f_{\bar{\theta}} \d  g$  on a chosen family of paths, is $g$-differentiable (since Theorem~\ref{thm:g-differentiable-yde} strongly uses a diagonal assumption on $g$).

We end this section by showing how the theory developed above applies to three specific classes of H\"older  maps.

\begin{example}[Sobolev maps]\label{ex:sobolev}
The theory developed above applies to $g \in W^{\ell,p}(I^m; \R^k)$, provided that  $\ell p > m$, so that Sobolev embedding gives $g \in C^{\beta}(I^m; \R^k)$ for $\beta = \min\cur{1, \ell - m/p}$. Let us notice that (e.g.\ when $m=2$, otherwise one has to use precise representatives for Sobolev functions) one can also define $\d g^i \wedge \d g^j$ by means of the weak Jacobian~\cite{brezis_2011} $J(g^i, g^j) := \det( \nabla g^i, \nabla g^j)$, 
\begin{equation} \int_{\partial Q} g ^i \d g^j = \int_{Q}  J(g^i, g^j) \d \mathscr{L}^2 \quad \text{ for any rectangle $Q \subseteq I^2$.}\end{equation}
Hence, condition $\d g^i \wedge \d g^j = 0$ in the form~\eqref{eq:wedge-null} is  equivalent to $\det( \nabla g^i, \nabla g^j) = 0$  $\mathscr{L}^2$-a.e.\ in $I^2$.
\end{example}

\begin{example}[random fields]
Let $(\Omega, \mathcal{A}, \mathbb{P})$ be a probability space and $g\colon \Omega \to C^{\beta}(I^m; \R^k)$ be a random field, with $\beta$-H\"older realizations for $\mathbb{P}$-a.e.\ $\omega$. Then, our results provide a ``pathwise'' calculus (i.e., for fixed $\omega$), in particular to solve systems of differential equations of Frobenius type. As an application of Theorem~\ref{thm:frob-2} one could think e.g.\ in case $m=2$, writing $(x,t) \in I^2$ as the evolution through time (possibly with a noise in time) of an ``irregular'' curve $\theta_{(\cdot, t)}$ and the resulting surface. A possible example of application could be in mathematical finance, concerning the modelling of yield curves $B(t,T)$ ($t \le T$) in the theory of interest rates, see e.g.~\cite{filipovic_2001}, which can be seen as a stochastic surface, with  ``noise''  acting on the direction $t$ only (however, usually modelled with an It\^o stochastic differential equation, whose H\"older regularity $<1/2$ puts it outside the scope of our results). Toy applications of Theorem~\ref{thm:frob-1} include e.g.\ the case of a single (scalar) signal, e.g.\ $g$ being a fractional Brownian sheet or a fractional Levy Brownian motion with Hurst parameter $H > 2/3$. Addressing integration of noise with lower regularity (e.g.\ the standard Brownian sheet or Levy Brownian motion) seems to require suitable adaptations of techniques from Rough Paths theory to multi-dimensional stochastic calculus.
\end{example}

\begin{example}[lacunary series]
Examples of H\"older functions can be constructed (on $I^m = [0,1]^m$) via Fourier series $g_p := \sum_{k \in \mathbb{Z}^m} a_k e^{i  2 \pi k \cdot p }$, e.g.\ by choosing only one coefficient $a_k$ different from zero in each annulus $2^{n} < |k| \le 2^{n+1}$ (and denote it by $c_n$). As with the classical Weierstrass function, if $\limsup_{n \to +\infty} |c_n| 2^{-n \beta} < \infty$, one has that $g \in C^{\beta}(I^m)$, and if $\liminf_{n \to +\infty} |c_n| 2^{-n \beta} >0$, then $g \notin C^{\beta'}(I^m)$, for any $\beta' >\beta$. These maps may be useful to provide counterexamples, showing e.g.\ that the regularity assumptions made throughout are necessary (as it is shown in~\cite{zust_integration_2011} for the problem of integrating forms). For example, in Proposition~\ref{prop:zust-1}, the assumptions can be stated even if $\alpha+\beta>1$ and $2 \beta>1$. An analysis of the proof gives that a counterexample could be given (if $\alpha+2\beta\le 2$) if one could find $g \in C^{\beta}(I^2)$, $v \in \jet^{\alpha}_g(I^2)$ such that $v g \notin \jet^{\alpha}_g(I^2)$, or more explicitly (after integrating by parts)
\[ \int_{\partial Q}  v \d g = 0 \quad \text{for every  rectangle $Q \subseteq I^2$,}
\]
but
\[  \int_{\partial Q} v \d g^2 \neq 0 \quad \text{for some  rectangle $Q \subseteq I^2$,}\]
which could be (formally) read as  failure of the chain rule for the solution $g$ to the ``continuity equation''
\[ \operatorname{div} (b g) = 0 \quad \text{in $I^2$,}\]
where $b = (\nabla v)^{\perp}$ (see~\cite{alberti_2014, bianchini_2016} for some recent literature on the subject of two-dimensional continuity equation and its renormalization properties). More rigorously, this is equivalent to exhibit two sequences $(g^n)_{n \ge 1}$, $(v^n)_{n \ge 1}$ of smooth functions, bounded respectively in $C^{\beta}(I^2)$ and $C^{\alpha}(I^2)$ such that, as $n \to \infty$, in the sense of distributions in $I^2$, one has
\[ \operatorname{div} ( (\nabla v^n)^{\perp} g ) \to 0, \quad \text{but} \quad\operatorname{div}  (\nabla v^n)^{\perp} g^2 ) \not \to 0. \]
Let us notice that, if such an example exists, then $(v,g)$ cannot be Sobolev (Example~\ref{ex:sobolev}), for the first identity would read $J(v,g) = 0$, $\mathscr{L}^2$-a.e.\ and the chain rule would give $J(v, g^2) = 2 g J(v, g) = 0$. We notice that in the case $\alpha=\beta$, this turns out to be equivalent to the problem of non-trivial horizontal surfaces in the Heisenberg group (using e.g.\ the results in \cite{zust_2015}), which has been solved  in \cite{wenger2018constructing}. 
\end{example}

\appendix

\section{Notation and useful results}\label{appendix-notation}


We recall and slightly extend to the multi-dimensional case some notation from the theory of rough paths as in~\cite{gubinelli_controlling_2004, friz_course_2014}.



\subsection*{Discrete differential calculus.}

Given a metric space $(X, \dist)$ and functions $f\colon X \to \R^k$, $\omega\colon X^2 \to \R^k$, we write $f(x) := f_x$, $\omega(x,y) := \omega_{xy}$ for $x$, $y \in X$ and let
\[ \delta f \colon X^2 \to \R^k, \quad \delta f_{xy} := f_y - f_x \quad \text{for $x$, $y \in X$,}\]
\[ \delta \omega \colon X^3 \to \R^k, \quad  \delta  \omega_{xyz} := \omega_{yz} - \omega_{xz} + \omega_{xy} \quad \text{for $x$, $y$, $z \in X$.}\]
Notice that $\delta( \delta f ) = 0$ and that discrete Leibniz rules hold in the following form. For $f$, $g \colon X \to \R$, $\omega\colon X^2 \to \R$, let $(fg)_x := f_x g_x$ and $(f\omega)_{xy} := f_x \omega_{xy}$, then
\begin{eqnarray}\nonumber \delta (fg)_{xy} & =  & (\delta f_{xy}) g _y + f_x (\delta g_{xy})\\
 & = & f_x (\delta g_{xy}) +  (\delta f_{xy}) g_{x} + (\delta f_{xy}) (\delta g_{xy}) \quad \text{for $x$, $y \in X$,} \label{eq:leibniz} \\
 \delta (f \omega)_{xyz} & =  & f_x (\delta \omega_{xyz}) +  (\delta f_{xy}) \omega_{yz} \quad \text{for $x$, $y$, $z \in X$.}\end{eqnarray}
In particular, letting $\omega = \delta g$, we have the identity
\begin{equation}\label{eq:delta-young} \delta (f \delta g)_{xyz} = (\delta f)_{xy} (\delta g)_{yz}, \quad \text{for $x$, $y$, $z \in X$.}\end{equation}

With a slight abuse of notation, when $X \subseteq \R^k$ and $f$ is the identity map, we write $\delta_{xy} = y-x$, hence $\dist(x,y) = |y-x| = | \delta_{xy}|$.

\subsection*{H\"older functions.} For a metric space $(X, \dist)$ and $f\colon X\to \R^k$, we let $[f]_0 := \sup_{x \in X} |f_x|$ and, for $\omega\colon X^2 \to \R^k$, $\alpha \ge 0$, we let
\[  [\omega]_\alpha := \sup_{\substack{x, y \in X \\ x \neq y} } \frac{ \abs{ \omega_{xy}} }{\dist(x,y)^\alpha } \in [0, +\infty].\]
Notice that $[\omega + \omega']_\alpha \le [\omega]_\alpha + [\omega']_\alpha$, $[f \omega]_\alpha \le [f]_0 [\omega]_\alpha$ and $[\omega]_\alpha \le [\omega]_\beta \diam(X)^{\beta-\alpha}$ if $\alpha \le \beta$. Moreover, for any  fixed $x \in X$, one has
\begin{equation}\label{eq:equivalence-holder-norms}  [f - f_x]_0 \le [\delta f]_\alpha \diam(X)^{\alpha}. \end{equation}
Write $f\in C^{\alpha}(X; \R^k)$ if $\norm{f}_\alpha := [f]_0 + [\delta f]_\alpha < \infty$. Notice that when $\alpha = 1$ one obtains the space of bounded Lipschitz functions on $X$ with values in $\R^k$, and not the usual space of continuous differentiable functions. If $k=1$ we simply write $C^{\alpha}(X) = C^{\alpha}(X; \R)$.

\subsection*{Young integration.} If $I \subseteq \R$ is a bounded interval, for $f \in C^{\alpha}(I; \R^{m \times d} )$, $g \in C^{\beta}(I; \R^{d \times n})$, with $\alpha + \beta >1$, L.C.\ Young~\cite{young_inequality_1936} provided a robust notion (i.e., extending continuously the case of smooth functions) to the integral
\[ \int_{a}^b f_s \d g_s, \quad \text{for $[a,b] \subseteq I$.}\]
Using the terminology introduced above, one can actually prove (as an application of the Sewing Lemma~\cite{gubinelli_controlling_2004, feyel_curvilinear_2006, friz_course_2014}) that, for some constant $\c = \c(\alpha, \beta)$, one has
\begin{equation}\label{eq:young-germ} \abs{\int_a^b f_s \d g_s - f_a (\delta g)_{ab} } \le \c [\delta f]_\alpha [\delta g]_\beta |\delta_{ab}|^{\alpha+\beta},\quad \text{for $[a,b] \subseteq I$,}\end{equation}
which gives the  inequality, again with $\c = \c(\alpha, \beta)$,
\begin{equation}\label{eq:young} \abs{\int_a^b f_s \d g_s} \le \c  \norm{f}_\alpha [\delta g]_\beta (1 + |I|^\alpha) |\delta_{ab}|^{\beta},\quad \text{for $[a,b] \subseteq I$,}\end{equation}
showing that the Young integral function $t \mapsto \int_a^t f_s \d g_s$ is 
 $C^{\beta}(I; \R^{m \times n})$. Moreover, additivity holds, for $a \le b \le c$, with $[a,c] \subseteq I$, 
\begin{equation}\label{eq:young-2} \int_a^c f_s \d g_s =\int_a^b f_s \d g_s + \int_b^c f_s \d g_s,\end{equation}
as well as bilinearity of $(f,g) \mapsto \int_a^b f \d g$.

\subsection*{Additive functionals.}
It is not difficult to prove  that the validity of~\eqref{eq:young} and~\eqref{eq:young-2} actually characterize the Young integral $\int f \d g$.  In fact, this can be seen as a consequence of a general result for (dyadically) additive functionals. For our purpose, we give it in the case of segments and rectangles (a special case of a general result on rectangles of dimension $k$). We say that a functional $F$, defined on oriented segments all contained in $\Omega$ is dyadically additive if  for every $[pq] \subseteq \Omega$,
\begin{equation}\label{eq:dyadically-additive-k1} F([pq]) = F( [p r ] ) + F([rq]), \quad \text{with $r = \frac{p+q}{2}$.} \end{equation}
Similarly, we say that a functional $F$ defined on oriented rectangles all contained in $\Omega$ is dyadically additive if,  for every $Q=(p; v_1, v_2) \subseteq \Omega$, 
\begin{equation}\label{eq:dyadically-additive-k2}\begin{split}  F\bra{ p; v_1, v_2 } & =   F\bra{ p; \frac{ v_1}{2}, \frac {v_2}{2}} + F\bra{ p + \frac {v_1}{2}; \frac{ v_1}{2}, \frac {v_2}{2}} \\
&  \quad + F\bra{ p + \frac {v_2}{2}; \frac{ v_1}{2}, \frac {v_2}{2}}+ F\bra{ p + \frac{v_1+v_2}{2}; \frac{ v_1}{2}, \frac {v_2}{2}}.\end{split}\end{equation}
We have the following result.

\begin{lemma}\label{lem:dyadic}
Let $k \in \cur{1,2}$, $\Omega \subseteq \R^d$ and $Q \mapsto F(Q) \in \R$ be defined on oriented segments (if $k=1$) or oriented rectangles (if $k=2$) $Q\subseteq \Omega$ and  dyadically additive. If $F(Q) =o\bra{ \diam(Q)^k}$, i.e., $F(Q) = 0$ if $\diam(Q) = 0$ and
\[ \inf_{\eps \to 0} \sup_{\substack{ Q \subseteq \Omega \\ \diam(Q) \ge \eps } } \frac{ |F(Q)| }{ \diam(Q)^k } = 0,\]
then $F$ is identically null.
\end{lemma}

\begin{proof}
Given $Q\subseteq \Omega$,  for any $n\ge 1$, we decompose it into $2^{kn}$ ``dyadic'' segments or rectangles $(Q_i)_{i=1}^{2^{kn}}$ (iterating the decompositions in ~\eqref{eq:dyadically-additive-k1} or~\eqref{eq:dyadically-additive-k2}) with $\diam(Q_i) =2^{-n} \diam (Q)$, for $i \in \cur{1, \ldots, 2^{kn}}$. By induction on the additivity assumption, we have
\[ |F(Q)| = \abs{\sum_{i=1}^{2^{kn}} F(Q_i)} \le  \sum_{i=1}^{2^{kn}} \abs{\diam(Q_i)^k} = 2^{kn} o\bra{ 2^{-kn}\diam(Q)^k } \to 0, \quad \text{as $n \to \infty$. \qedhere}\]
\end{proof}

To obtain from this result the claimed characterization of the Young integral, assume that $\int f \d g $ and $\int' f \d g$ both satisfy~\eqref{eq:young} and~\eqref{eq:young-2}. Then, letting $F := \int f \d g - \int' f \d g$,~\eqref{eq:young-2} yields that $F$ is dyadically additive, while adding and subtracting $f_a(\delta g)_{ab}$ in~\eqref{eq:young} gives 
\[ |F([a,b]) | \le \c [\delta f]_\alpha [\delta g]_\beta |\delta_{ab}|^{\alpha+\beta} = o( |\delta_{ab}|),\]
hence $F([a,b]) =0$ and the two integrals coincide.

\begin{remark}[from approximate to actual identities] \label{rem:from-approximate-to-identity} A slight extension of the same argument gives the following result. Let $f \in C(I)$, let $n \ge 1$, and for $i \in \cur{1, \ldots, n}$, let $\alpha_i$, $\beta_i  \in (0,1]$ with $\alpha_i+ \beta_i >1$, $f^i \in C^{\alpha_i}(I)$, $g ^i \in C^{\beta_i}(I)$ such that
\begin{equation}\label{eq:approximate-identity} \delta f_{ab} = \sum_{i=1}^n f_a^i \delta g^i_{ab} + o( |\delta_{ab}| )\quad \text{for $[a,b] \subseteq I$.}\end{equation}
Then, by considering $F :=  \delta f_{ab} - \sum_{i=1}^n  \int f^i \d g^i$, we obtain
\begin{equation}\label{eq:true-identity} \delta f_{ab} = \sum_{i=1}^n  \int_{a}^b f^i_s \d g^i_s \quad \text{for $[a,b] \subseteq I$.}\end{equation}
\end{remark}

\subsection*{Integration over curves.} Young integration allows one to extend the notion of integration of $1$-forms $\omega = f \d g $ along curves $\gamma\colon I \subseteq \R \to \Omega \subseteq \R^d$, provided that the regularity requirements are satisfied. For example, if $f \in C^{\alpha}(\Omega; \R)$, $g \in C^{\beta}(\Omega; \R)$, with $\alpha+\beta>1$ and $\gamma \in C^1( I; \Omega)$, then $f \circ \gamma \in C^{\alpha}(I; \R)$, $g \circ \gamma \in C^{\beta}(I; \R)$ and so we define
\[ \int_{\gamma } f \d g := \int_I f\circ \gamma \, \d g \circ \gamma. \]

We prove that, as in the smooth case, such integral does not depend on the parametrization, with the exception of the orientation. Precisely, if $\varphi \in C^1(J; \Omega)$ and there exists $\theta \in C^1( J; I)$ such that $\varphi = \gamma \circ \theta$, then
\[ \int_{\varphi } f \d g = \int_{\gamma } f \d g.\]
This can be seen at least in two ways: either using the fact that the same identity holds in the smooth case and approximating $f$ and $g$, or using the characterization of Young integral by~\eqref{eq:young} and~\eqref{eq:young-2}: indeed, the functional
\[ [a,b]\subseteq J \mapsto \int_{\gamma \res [\theta_a,\theta_b] } f \d g  = \int_{\theta_a}^{\theta_b} f\circ \gamma \, \d g \circ \gamma,\]
satisfies the additivity condition~\eqref{eq:young-2} and, by~\eqref{eq:young-germ}, the inequality, for $[a,b]\subseteq J$,
\[\begin{split} 
\abs{\int_{\gamma \res [\theta_a,\theta_b] } f \d g - (f \circ \varphi)_{a} \bra{ \delta (g \circ \varphi)}_{a b} } & =  \abs{ \int_{\theta_a}^{\theta_b} f\circ \gamma\, \d g \circ \gamma - (f \circ \gamma)_{\theta_a} \bra{ \delta (g \circ \gamma)}_{\theta_a \theta_b} }\\
&  \le \c | \delta \theta_{ab}|^{\alpha+ \beta} \le \c [ \delta \theta]_1^{\alpha+\beta} |\delta_{ab}|^{\alpha+ \beta},\end{split}\]
where $\c = \c(\alpha, \beta)[\delta (f\circ \gamma)]_{\alpha} [ \delta (g\circ \gamma)]_\beta$.

When $\gamma_t = (1-t) p + t q$ parametrizes the (oriented) segment $[pq]$, we write
\[ \int_{[pq]} f \d g := \int_{ \gamma } f \d g.\]
Notice that, given $p$, $q \in \Omega$ (with the segment $[pq]$ contained in $\Omega$), one has
\begin{equation}\label{eq:integral-change-orientation}
\int_{[pq]} f \d g = - \int_{[qp]} f \d g\end{equation}
and, whenever $p$, $q$, $r$ are collinear,
\begin{equation}\label{eq:split-integral-line}
\int_{[pq]} f \d g = \int_{[pr]} f \d g + \int_{[rq]} f \d g
\end{equation}
because of~\eqref{eq:young-2} and the fact that one can consider a common parametrization to compute the three integrals. Moreover,~\eqref{eq:young} gives, for $\c = \c(\alpha, \beta)$,
\begin{equation}\label{eq:young-curves-germs} \abs{\int_{[pq]} f \d g - f_p (\delta g)_{pq} } \le \c [\delta f]_\alpha [\delta g]_\beta |\delta_{pq}|^{\alpha+\beta},\quad \text{for $p$, $q$ with $[pq]\subseteq \Omega$.}\end{equation}
 Finally, Remark~\ref{rem:from-approximate-to-identity} extends to the case of segments: from an identity of the type~\eqref{eq:approximate-identity}, valid for every $a$, $b \in \Omega$ with $[ab] \subseteq \Omega$, we deduce that~\eqref{eq:true-identity} holds true. 

\subsection*{Integration over boundaries.}
Given an oriented rectangle  $Q=[p; v_1, v_2]$ contained in $\Omega$,  the integral of $f \d g$ on the boundary of $Q$ is defined as
\[ \int_{\partial Q} f \d g := \int_{[p (p+v_1)]} f \d g + \int_{[(p+v_1) (p+v_1+v_2)]} f \d g +\int_{[(p+v_1+v_2) (p+v_2)]} f \d g + \int_{[(p+v_2)p]} f \d g.\]
Because of~\eqref{eq:integral-change-orientation} and~\eqref{eq:split-integral-line}, the map  $Q \mapsto \int_{\partial Q} f \d g$ is dyadically additive. Other simple properties, such as bi-linearity of $(f,g) \mapsto \int_{\partial Q} f \d g$ follow from those of Young integrals. Moreover, we have the inequality
\begin{equation}\label{eq:young-boundary} \begin{split}
\abs{ \int_{\partial Q} f \d g } & = \abs{ \int_{\partial Q} (f-f_p) \d g + \int_{\partial Q} f_{p} \d g} \quad \quad \text{for any $p \in Q$,}\\
& = \abs{ \int_{\partial Q} (f-f_p) \d g} \quad \quad \text{since $\int_{\partial Q} f_p \d g = f_p \int_{\partial Q} 1 \d g = 0$,}\\
&  \le \c \bra{ [f-f_p]_0 [\delta g]_\beta \diam(Q)^\beta + [\delta (f-f_p)]_\alpha [\delta g]_\beta \diam(Q)^{\alpha+\beta} } \quad \text{by~\eqref{eq:young},}\\
&  \le 2 \c [\delta f ]_\alpha [\delta g]_\beta \diam(Q)^{\alpha +\beta}  \quad\quad \text{by~\eqref{eq:equivalence-holder-norms},}
\end{split}\end{equation}
where $\c = \c(\alpha, \beta)$.

%

\begin{remark}[Z\"ust integral]
By formally applying Stokes' theorem, one has
\[ \int_{\partial Q} f \d g = \int_Q 1 \d f \wedge \d g.\]
In fact, this can be made rigorous by extending Young integration to $k$-forms, as done by in  R.\ Z\"ust~\cite{zust_integration_2011}, providing a robust notion to the integral
\[ \int_Q f \d g^1 \wedge \ldots \wedge \d g^k \]
for rectangles $Q \subseteq \Omega$ and $f\in C^{\alpha}(\Omega)$, $g^i \in C^{\beta_i}(\Omega)$, with $\alpha+ \sum_{i=1}^k \beta_i > k$. However, for our purpose, we do not need to rely on this theory.
\end{remark}


%
%
%
%
%
%
%

\section{Jets and $g$-differentiable maps}\label{appendix-jets}

\begin{proof}[Proof of Theorem~\ref{thm:jets}]

{\em Step 1}. We first show that $v=\diff_g\theta$ for some $\theta\in C(I^m;\R^{d})$, if and only if~\eqref{eq:g-jet} holds. 
To this aim assume first that $\theta \in C(I^m, \R^d)$ is $g$-differentiable with $v = \diff_g \theta$. By applying Remark~\ref{rem:from-approximate-to-identity} to the ``approximate  identity''~\eqref{eq:g-differentiable}  we obtain that~\eqref{eq:integration-g-diff} holds. Then, summing over the four oriented edges of a rectangle $Q \subseteq I^m$, the left hand side terms cancel out, providing~\eqref{eq:g-jet}, i.e., $v \in \jet^{\alpha}_g(I^m; \R^{d \times k})$.

Conversely, assuming that $v \in \jet^{\alpha}_g(I^m; \R^{d \times k})$, we construct a $\theta\in C(I^m;\R^{d})$ by the formula
\[
\theta_p:=\theta_0+\sum_{i=1}^m \int_{[\bar p_{i-1}\bar p_i]} v \d g, 
\] 
where
$\bar p_i:=(p_1,\ldots, p_i, 0,\ldots, 0)$, $\bar p_0:=0$, or in other words,
\[
\theta_p=\theta_0+\int_{\gamma^p} v \d g, 
\]
the curve $\gamma^p$ standing for a parameterization
of the polygonal line $\bar p_0\bar p_1\ldots\bar p_m$. To show the 
continuity of $\theta$ together with~\eqref{eq:g-differentiable}, we write
\[
\theta_q=\theta_0+\int_{\gamma^q} v \d g.
\] 
Then setting $\gamma^{pq}:= p +\gamma^{q-p}$,
$\gamma:=\gamma^p\cdot \gamma^{pq}\cdot (-\gamma^q)$, $\cdot$ standing for the concatenation of curves, parameterized for convenience still over $[0,1]$,
$-\gamma_q$ standing for the curve with the same trace as $\gamma_q$ and opposite direction, we have that $\gamma$ is a closed curve and
\[
\delta\theta_{pq} = \int_\gamma v \d g + \int_{\gamma^{pq}} v \d g.
\]
Clearly, $\gamma$ as a $1$-chain can be viewed as a finite sum of the boundaries
of the $2$-chains associated to rectangles in $I^m$ with  sides parallel
to coordinate axes, and therefore, the first integral in the right-hand side
of the above equality vanishes, so that
\[
\delta\theta_{pq} =  \int_{\gamma^{pq}} v \d g =\int_0^{\ell(\gamma^{pq})} v\circ\gamma^{pq} \d g\circ\gamma^{pq},
\]
where in the last integral one takes the arclength parameterization of
 $\gamma^{pq}$ over $[0,\ell(\gamma^{pq})]$.
Then
\begin{align*}
|\delta\theta_{pq}| &=  (v\circ\gamma^{pq})_0 (\delta g\circ\gamma^{pq})_{0\ell(\gamma^{pq})} + O(\ell(\gamma^{pq})^{\alpha+\beta})\\
& = v_p \delta g_{pq} + O(|\delta_{pq}|)^{\alpha+\beta})\\
& = v_p \delta g_{pq} + o(\delta_{pq})
\end{align*}
proving the 
continuity of $\theta$ together with~\eqref{eq:g-differentiable}.

{\em Step 2}. We show now that~\eqref{eq:g-jet} 
is equivalent to~\eqref{eq:g-jet1gam}. In fact, as proven in Step~1,~\eqref{eq:g-jet} 
is equivalent to
the existence of a $\theta\in C(I^m;\R^{d})$ such that $v=\diff_g\theta$.
The claim follows then from Lemma~\ref{lm_int_gdiff1} below applied
(with $a:=0$, $b:=1$) to
closed Lipschitz curves $\gamma\colon [0,1] \to I^m$.
As a byproduct we have that
\[
\theta_q=\theta_0+\int_{\theta} v \d g
\]
for any Lipschitz $\theta\colon [0,1] \to I^m$ with $\theta(0)=p$, $\theta(1)=q$, in particular, for $\theta$ standing for a parametrization of the segment $[pq]$, proving~\eqref{eq:integration-g-diff}.

{\em Step 3}. Finally  \eqref{eq:integration-g-jet-2} implies the bound
\[
 [\delta \theta]_{\beta} \le \bra{ [\diff_g \theta ]_0 + \c [\delta (\nabla _g \theta)]_\alpha |I|^{\alpha} } [\delta g]_\beta.
\]
from which the last statement of the thesis of Theorem~\ref{thm:jets} 
follows immediately. \end{proof}

\begin{lemma}\label{lm_int_gdiff1}
Let $\alpha$, $\beta \in (0,1]$, with $\alpha+\beta>1$, 
$\theta \in C(I^m; \R^d)$ be $g$-differentiable with $g$-derivative
$\diff_g \theta\in C^{\alpha}(I^m; \R^{d\times k})$, and 
$g \in C^{\beta}(I^m; \R^k)$.
Then
for every $\gamma\colon [0,1] \to I^m$ Lipschitz
and for every $a,b\in [0,1]$, $a\leq b$, one has
\begin{equation}\label{eq:g-jet1gam2}
(\delta\theta\circ\gamma)_{ab}=
 \int_a^b  \diff_g \theta (\gamma(t)) \d g(\gamma(t)).  
\end{equation}
\end{lemma}

\begin{proof}
By the basic estimate on Young integral we have 
\begin{align*}
\int_a^b  \diff_g \theta (\gamma(t)) \d g(\gamma(t)) &=
(\diff_g \theta\circ\gamma)_a \delta (g\circ \gamma)_{ab} + O(|\delta_{ab}|^{\alpha+\beta}) \\
& = (\diff_g \theta\circ\gamma)_a \delta (g\circ \gamma)_{ab} + o(\delta_{ab})
\quad\text{because $\alpha+\beta>1$},
\end{align*}
hence from~\eqref{eq:g-differentiable} we get
\[
\int_a^b  \diff_g \theta (\gamma(t)) \d g(\gamma(t)) = (\delta\theta\circ\gamma)_{ab} + o(\delta_{ab}),
\]
and since germs on the right and left-hand sides of~\eqref{eq:g-jet1gam2} are dyadically additive on $[0,1]$, this gives~\eqref{eq:g-jet1gam2} by Lemma~\ref{lem:dyadic}.
\end{proof}

\begin{proposition}[anisotropic chain rule]\label{prop:chain-rule}
Let $\alpha \in (0,1]$, and, for $i \in \cur{1,2}$, let $\beta_i, \gamma_i \in (0,1]$, $g^i, \theta^i \in C^{\beta_i}(I^{m_i} ; J^{n_i})$, $h^i \in C^{\gamma_i}( J^{n_i} ; \R^{k_i})$. Let $m := m_1+m_2$,  $I^m = I^{m_1} \times I^{m_2}$, $p =(p^1, p^2) \in I^m$ and $g_p := (g^1_{p^1}, g^2_{p^2})$, and similarly $n :=n_1+n_2$,  $J^n = J^{n_1} \times J^{n_2}$, $x =(x^1, x^2) \in J^n$, $h_x := (h^1_{x^1}, h^2_{x^2})$. Assume that
\begin{equation}\label{eq:condition-chain-rule}  \min\cur{\beta_1\gamma_1, \beta_2 \gamma_2} + \alpha \min\cur{\gamma_1, \gamma_2} >1.\end{equation}
If $f:J^n \to \R^{d}$ is $h$-differentiable with $\diff_h f \in C^{\alpha}(J^n; \R^{d \times (k_1+k_2)})$ and $h\circ\theta\colon I^m \to J^n$ is $g$-differentiable, then $f \circ \theta$ is $g$-differentiable with
\[  \diff_g (f \circ \theta) = (\diff_h f)_{\theta} \diff_g (h\circ \theta).\]
\end{proposition}

\begin{proof}
First, notice that~\eqref{eq:condition-chain-rule} implies $\alpha+ \gamma_i>1$ for $i \in \cur{1,2}$, hence Theorem~\ref{thm:jets} applies with $f$ in place of $\theta$ and $h$ in place of $g$ and $\gamma$ instead of $\beta$, yielding the identities
\[  \delta f_{xy} = \int_{[xy]}  \diff_h f \d h = \int_{[xy]}  \diff_h f^1 \d h^1 + \int_{[xy]}  \diff_h f^2 \d h^2.\]
Estimating separately the two integrals above by means of~\eqref{eq:young-curves-germs}, we deduce that the expansion
\begin{equation}\label{eq:expansion-chain-rule-1} \delta f_{xy}  = (\diff_h f)_x \cdot \delta h_{xy} + O\bra{ |\delta_{xy}|^{\alpha} (|\delta_{x^1 y^1}|^{\gamma_1} + |\delta_{x^2 y^2}|^{\gamma_2})  }  \end{equation}
holds true. Choosing $x = \theta_p$, $y = \theta_q$, one has, for $i \in \cur{1,2}$,
\[ | \delta_{x^i y^i} | = | \delta \theta^i_{pq} | \le [\delta \theta^i]_{\beta_i} | \delta_{pq}|^{\beta_i},\]
hence the second term in the  right hand side in~\eqref{eq:expansion-chain-rule-1} is $o(\delta_{pq})$, because of~\eqref{eq:condition-chain-rule}. For the first term, since $h \circ \theta$ is $g$-differentiable, we have
\[ (\diff_h f)_{\theta_p} \cdot \delta h_{\theta_p \theta_q} = (\diff_h f)_{\theta_p} (\diff_g h\circ \theta)_p + o(\delta_{pq})\]
and the thesis follows.
\end{proof}

\begin{proof}[Proof of Proposition~\ref{prop:zust-1}]
We introduce the function $F$, defined on rectangles $Q \subseteq I^m$,
\[  F(Q) := \int_{\partial Q} v \d g.\]
As seen in Appendix~\ref{appendix-notation}, $F$ is dyadically additive, hence to show  that $F$ is null it is sufficient by  Lemma~\ref{lem:dyadic} to prove that that $F(Q) = o( \diam(Q)^2)$, as $\diam(Q) \to 0$. 

Fix $Q \subseteq I^m$ and $\bar{p} \in Q$. Notice that we can always assume that $g_{\bar{p}} = 0$, since replacing $g$ with $g-g_{\bar{p}}$ leaves $\delta g$ unchanged, hence the integral defining $F(Q)$, as well as $\diff_g v$ ($g$-differentiability depends on $g$ only through its increments). Moreover, we may restrict our analysis from $I^m$ to $Q$, so that the inequality $[g]_0 \le [\delta g]_\beta \diam(Q)^\beta$  holds. Integrating by parts, we have
\begin{equation} F(Q) =  \int_{\partial Q}  v \d g = -\int_{\partial Q}  g \d v. \end{equation}
As a consequence of the $g$-differentiability assumption, we  claim that the following identity holds:
\begin{equation}\label{eq:stokes-without-zust}
\int_{\partial Q}  g \d v =  \int_{\partial Q}  g (\diff_g v) \d g = \sum_{i,j=1}^k \int_{\partial Q}  g^i (\partial_{g^j} v^i) \d g^j.
\end{equation}
Once this identity is established, the thesis follows by proving that, for any $i, j \in \cur{1, \ldots, k}$
\begin{equation}\label{eq:thesis-germs} \int_{\partial Q}  g^i (\partial_{g^j} v^i) \d g^j + \int_{\partial Q}  g^j (\partial_{g^i} v^j) \d g^i =  o( \diam(Q)^2 ).\end{equation}
The key observation to prove~\eqref{eq:thesis-germs} is that
\begin{equation}\label{eq:stokes-zero} (\partial_{g^j} v^i)_{\bar{p} } \int_{\partial Q}  g^i  \d g^j + (\partial_{g^i} v^j)_{\bar{p} } \int_{\partial Q}  g^j  \d g^i = 0,\end{equation}
as a consequence of the hypothesis. Indeed, if $\d g^i \wedge \d g^j = 0$ holds in the form~\eqref{eq:wedge-null}, then both integrals in~\eqref{eq:stokes-zero} are zero, otherwise we have $(\partial_{g^j} v^i)_{\bar{p}}  =(\partial_{g^i} v^j)_{\bar{p}}$, and integration by parts gives
\[ \int_{\partial Q}  g^j  \d g^i =  \int_{\partial Q}  \d (g^j   g^i) - \int_{\partial Q}  g^i \d g^j =- \int_{\partial Q}  g^i \d g^j.\]
Therefore, subtracting~\eqref{eq:stokes-zero} from~\eqref{eq:thesis-germs}, by linearity of Young integral, 
 to prove~\eqref{eq:thesis-germs}
it is suffices to show 
\[  \int_{\partial Q}  g^i \bra{ (\partial_{g^j} v^i) -  (\partial_{g^j} v^i)_{\bar{p} } } \d g^j + \int_{\partial Q}  g^j \bra{ (\partial_{g^i} v^j) -   (\partial_{g^j} v^i)_{\bar{p} } } \d g^i   = o( \diam(Q)^2 ).\]
In fact, we estimate the two integrals above separately (and we argue only with the first one, the second being similar). Using~\eqref{eq:young-boundary} with $\gamma := \min\cur{\alpha, \beta}$ instead of $\alpha$, we have, with $\c = \c(\gamma, \beta)$,
\[ \begin{split} &  \abs{ \int_{\partial Q}  g^i \bra{ (\partial_{g^j} v^i) -  (\partial_{g^j} v^i)_{\bar{p} } } \d g^j  }  \c \le [ \delta \bra{  g^i \bra{ (\partial_{g^j} v^i) -  (\partial_{g^j} v^i)_{\bar{p} } }} ]_{\gamma} [\delta g^j]_{ \beta} \diam(Q)^{\gamma+\beta} \\
& \quad \le \c \bra{ [ \delta  g^i]_\gamma [\partial_{g^j} v^i -  (\partial_{g^j} v^i)_{\bar{p} }  ]_{0} + [ g^i]_0 [ \delta (\partial_{g^j} v^i)  ]_{\gamma} }[\delta g^j]_{ \beta}\diam(Q)^{\gamma+ \beta}\\
& \quad \le 2 \c [ \delta  g^i]_\beta [\delta g^j]_{ \beta} [\delta (\partial_{g^j} v^i) ]_{\alpha} \diam(Q)^{\alpha+2\beta}, \ 
\end{split}\]
using also the inequality $[g^i]_{0} \le [\delta g^i]_{\beta}\diam(Q)^{\beta}$,  to hence~\eqref{eq:thesis-germs}.

Finally, to prove~\eqref{eq:stokes-without-zust}, it is sufficient to show that, for $p$, $q \in Q$,
\[ \int_{[pq]} g \d v =  \int_{[pq]} g(\diff_g v) \d g.\]
In turn, the identity between the two integrals follows, by Remark \ref{rem:from-approximate-to-identity}, from the validity of the ``approximate  identity''
\[ g_p \delta v_{pq} = g_p (\diff_g v)_p \delta g_{pq} + o(\delta_{pq}),\]
which can be obtained multiplying the expansion in the definition of $g$-differentiability~\eqref{eq:g-differentiable} for $v$ by the uniformly bounded function $g_p$. 
\end{proof} 

\begin{proof}[Proof of Lemma~\ref{lem:examples-dg-dg-null}]
First, we prove that, for any $i$, $j \in \cur{1, \ldots, n}$,~\eqref{eq:wedge-null} holds with $h$ instead of $g$. We introduce the dyadically additive function 
\[ F(Q) := \int_{\partial Q} h^i \d h^j, \quad \text{for rectangles $Q \subseteq I^m$.}\]
and prove that $F(Q) = o(\diam(Q)^2)$. For a fixed rectangle $Q \subseteq I^m$, choosing any $\bar{p} \in Q$, by subtracting $\int_{\partial Q} h^i_{\bar{p}} \d h^j = 0$, we can assume that $h^i_{\bar{p}} = 0$ and, restricting the argument to $Q$ instead of $I^m$, the bound $[ h^i]_0 \le [\delta h^i]_{\beta} \diam(Q)^{\beta}$ holds. Moreover, we can also replace $h^j$ with $h^j-h^j_{\bar{p}}$ and similarly $w$ with $w- w_{\bar{p}}$ so that $[ h^j]_0 \le [\delta h^j]_{\beta} \diam(Q)^{\beta}$ and $[w]_0 \le [\delta w]_{\beta} \diam(Q)^{\beta}$. Arguing as in the proof of~\eqref{eq:stokes-without-zust}, we have the identity
\[ \int_{\partial Q} h^i \d h^j = \int_{\partial Q} h^i (\partial_w h^j) \d w \]
By~\eqref{eq:young-boundary}, with $\varepsilon= \min\cur{\alpha, \beta}$ in place of $\alpha$, we have
\[ \begin{split} \abs{ \int_{\partial Q} h^i \bra{ (\partial_w h^j)- (\partial_w h^j)_{\bar{p}} } \d w }  & \le \c \sqa{\delta \bra{ h^i \bra{ (\partial_w h^j)- (\partial_w h^j)_{\bar{p}} }} }_{\varepsilon} [ \delta w]_{\beta} \diam(Q)^{\varepsilon+\beta} \\
& \le \c [\delta h^i]_{\beta}[\delta \partial_w h^j]_{\alpha} [\delta w]_{\beta} \diam(Q)^{\alpha+ 2 \beta} = o(\diam(Q)^2), \end{split}\]
hence it is sufficient to prove that
\[  \int_{\partial Q} h^i (\partial_w h^j)_{\bar{p}} \d w = (\partial_w h^j)_{\bar{p}} \int_{\partial Q} h^i \d w =  o(\diam(Q)^2).\]
Actually, Proposition~\ref{prop:zust-1} gives that $h^i \in \jet^{\alpha}_w(I^m; \R)$ since~\eqref{eq:wedge-null} holds for $g^i = g^j =w$, hence we already have that $\int_{\partial Q} h^i \d w  =0$.

Let then $f \in C^{\gamma}(\R^n; \R^k)$ be in the assumptions. To show that $g:= f \circ h$ satisfy~\eqref{eq:wedge-null}, we use stability with respect to approximations of $f$. Precisely, let $(f^i)_{i \ge 1}$ be a sequence of smooth maps converging to $f$ (locally) in $C^{\gamma}(\R^n; \R^k)$, obtained e.g.\ by convolution. The chain rule (Proposition~\ref{prop:chain-rule}) gives that $f^i \circ h$ are all $w$-differentiable, hence the thesis holds by the previous discussion. On the other side, $f^i \circ h$ converge to $f \circ h$ in $C^{\beta \gamma}(I^m; \R^k)$, and continuity of Young integral in this topology (due to the assumption $2 \beta\gamma >1$ ensure that the thesis holds in the limit as $i \to +\infty$ as well.
\end{proof}

\section{Proof of the implicit function theorem}\label{appendix-dini}

\begin{proof}[Proof of Theorem~\ref{thm:dini}]
As in the classical implicit function theorem, existence of $\theta$ is established by a fixed point argument.  To simplify notation, let us assume that $x_0 = 0 \in \R^{m+n}$. For $x$, $y \in I^{m+n}$, we let 
\begin{equation} \label{eq:taylor-dini} \varrho_{xy} := \delta f_{xy} - (\diff_{h} f)_x \delta h_{xy},\end{equation}
be the remainder of the Taylor expansion of $f$ at $x$ in terms of $h$, so that Proposition~\ref{prop:chain-rule}, with $\beta_1 = \beta$ and $\beta_2 = 1$, gives
\begin{equation}\label{eq:estimate-remainder} |\varrho_{xy} | \le \c [\delta (\diff_h f) ]_{\alpha} \bra{ [\delta g]_{\beta} |\delta_{x^m y^m}|^{\beta} + |\delta_{x^n y^n}| }  |\delta_{xy}|^{\alpha},
\end{equation}
with $\c = \c(\alpha, \beta)$. Subtracting~\eqref{eq:taylor-dini} with $(0, x)$ and $(0, y)$ instead of $(x,y)$ one  obtains the identity
\[ \delta f_{xy} =  (\diff_{h} f)_{0} \delta h_{xy} + \varrho_{0 y}- \varrho_{0 x}, \]
From which we  deduce that $\delta f_{xy} = 0$ if and only if
\[ 0 = (\diff_{x^n} f)_{0} \delta_{x^n y^v} + (\diff_{g} f)_{0} \delta g_{x^m y^m} + \bra{ \varrho_{0 y}- \varrho_{0 x} }, \]
i.e., writing $r_{x} := (\diff_{x^n} f)_{0}^{-1}\varrho_{0 x}$, the following equation holds:
\begin{equation}\label{eq:lsde} \delta_{x^n y^n} =  -(\diff_{x^n} f)_{0}^{-1} \delta g_{x^m y^m}  - \delta r_{xy}.\end{equation}
We also notice that, starting from the identity
\[  \begin{split} \varrho_{0 y} - \varrho_{0 x} & = \varrho_{xy} - \delta \varrho_{0 x y} \\
& = \varrho_{xy} - \delta (\diff_{h} f)_{0 x} \delta h_{xy}
\end{split}\]
we obtain, after multiplication with $(\diff_{x^n} f)^{-1}_{0}$,
\begin{equation}\label{eq:estimate-r}
\begin{split} | \delta r_{xy} |  &  \le | (\diff_{x^n} f)^{-1}_{0} |\bra{  | \varrho_{xy} | +  |\delta (\diff_{h} f)_{0 x}| | \delta h_{xy} | }\\
& \le  \c  \bra{ |\delta_{xy}|^{\alpha} + |\delta_{0 x}|^{\alpha} }    \bra{ |\delta_{x^m y^m}|^{\beta} + |\delta_{x^n y^n}|  }  \quad \text{using~\eqref{eq:estimate-remainder},}\end{split}
\end{equation}
where $\c =  \c ( \alpha, \beta, f, g)$.

 We introduce a interval $J$ with $\bar{J} \subseteq I$ with length $|J|$ to be specified later  and such that $0 \in J^m$. On the set
\[ V := \cur{ \theta \in C^{\beta}(\bar{J}^m; I^n) \colon \theta_{0} = 0, \,  [\delta \theta]_\beta \le \eps },\]
we introduce the map
\[ F( \theta) :=   -(\diff_{x^n} f)_{0}^{-1} g - r_{\bar{\theta}},\]
so that any fixed point $\theta \in V$ for $F$ yields $\theta$ such that $f_{\bar{\theta}} = f_{0}$, where $\bar{\theta}_p :=(p, \theta_p)$, $p \in J^m$. To show existence of fixed points, notice first that
\begin{equation}\label{eq:delta-F-dini} \delta F( \theta)_{pq} = -(\diff_{x^n} f)_{0}^{-1}\delta g_{pq} - \delta r_{\bar{\theta}_p \bar{\theta}_q},\end{equation}
so that by~\eqref{eq:estimate-r},
\[ \begin{split} | \delta F(\theta)_{pq} | & \le  | (\diff_{x^n} f)_{0}^{-1} | |\delta g_{pq}| + |\delta r_{\bar{\theta}_p \bar{\theta}_q}| \\
& \le  \c \bra{ | \delta _{pq} |^\beta  + \diam(J^m)^{\alpha} |\delta_{pq}|^{\beta} + \diam(J^m)^{\alpha} |\delta_{\theta_p \theta_q}|} \quad \text{by~\eqref{eq:estimate-r},}\\
& \le \c \bra{ 1+  \diam(J^m)^{\alpha} [\delta \theta]_{\beta}}  | \delta _{pq} |^\beta  \le \eps  | \delta _{pq} |^\beta,
\end{split}\]
provided that $J$ and $\eps$ are chosen such that
\begin{equation}\label{eq:condition-fixed-point-dini-1} \c \bra{ 1+  \diam(J^m)^{\alpha} \eps } \le \eps.\end{equation}
To show that $F:V \to V$ is a contraction ($V$ being endowed with the uniform norm) given $\theta$, $\varphi \in V$, one has
\[ \begin{split}  | F(\varphi)_p - F(\theta)_p |  &  =  | r_{\bar{\theta}_p \bar{\varphi}_p} |\\
& \le \c \bra{ |\delta_{\bar \theta_p \bar \varphi_p }|^{\alpha} + |\delta_{0 \bar{\theta}_p}|^{\alpha} }   |\delta_{\theta_p \varphi_p}| \quad \text{by~\eqref{eq:estimate-r},} \\
& \le \c  \eps^\alpha \diam(J^m)^{\beta\alpha} [\theta - \varphi]_0 \le \frac 1 2  [\theta - \varphi]_0, \end{split} \]
provided that $J$ and $\eps$ are chosen such that
\begin{equation}\label{eq:condition-fixed-point-dini-2}  \c  \eps^\alpha \diam(J^m)^{\beta\alpha} \le \frac 1 2. \end{equation}
Therefore, if~\eqref{eq:condition-fixed-point-dini-1} and~\eqref{eq:condition-fixed-point-dini-2} are satisfied, there exists a unique $\theta \in V$ such that $F(\theta) = \theta$, and in particular $f_{\bar{\theta}} = f_{0}$.

To show that $\bar{\theta}$ is surjective on the level set $f^{-1}(f_{0})$ (possibly up to choosing a smaller $J$), let $x =(x^m, x^n)\in J^{m} \times J^{n}$ with $f_x = f_{0}$ and choose $p = x^m$, so that~\eqref{eq:lsde} with $y = \bar{\theta}_p$ gives
\[ \begin{split} |\delta_{x^n \theta_p}| & =   |\delta r_{x \bar{\theta}_p} |  \le \c \diam(J^m)^{\alpha} | \delta_{x^n \theta_p}|  \quad \text{by ~\eqref{eq:estimate-r},} \\
& < | \delta_{x^n \theta_p}|\end{split}\]
provided that $J$ is such that $\c | J|^{\alpha} < 1$. This yields a contradiction, unless $x^n = \theta_p$.

Finally, to show that $\theta$ is $g$-differentiable, we have, from~\eqref{eq:taylor-dini} with $x = \bar \theta_p$, $y =\bar  \theta_q$, 
 \[ \varrho_{\bar \theta _p \bar \theta _q} = - (\partial_{g} f)_{\bar{\theta}_p } \delta g_{pq} - (\diff_{x^n} f)_{\bar{\theta}_p } \delta \theta_{pq},\]
hence~\eqref{eq:g-dini} follows, since~\eqref{eq:estimate-remainder} gives
\[  \varrho_{\bar \theta _p \bar \theta _q}  =  O(| \delta_{pq}|^{\beta(1+\alpha)} ),\]
and we can multiply both sides with $(\diff_{x^n} f)_{\bar{\theta}_p }^{-1}$, which is everywhere invertible on $J$ (possibly choosing a smaller $J$) by continuity of  $\diff_{x^n} f$ and $\bar{\theta}$. \end{proof}

\begin{proof}[Proof of Proposition~\ref{prop:g-induced-composition}]
First, write
\[ \delta \varphi_{xy} = (\diff_f \varphi)_x  \delta f_{xy} + O( |\delta_{xy} |^{(\gamma+ \beta)} )\]
and then let $x = \bar{\theta}_p$, $y = \bar{\theta}_q$, with $\theta$ as in Theorem~\ref{thm:dini}, so that
\[  \delta \varphi_{\bar{\theta}_p \bar{\theta}_q} = O( |\delta_{\bar{\theta}_p \bar{\theta}_q} |^{(\gamma+ \beta)} ) = O( |\delta_{pq}|^{(\gamma+ \beta)\beta}) = o(\delta_{pq}),\]
for $p$, $q \in J^m$. We deduce that $\varphi_{\bar{\theta}_p}$ is constant, i.e., $z$ is (locally) constant on the level sets of $f$, hence we may represent $\varphi = \Phi \circ f$.  
\end{proof}

\section{Proof of Frobenius theorems}\label{appendix-frobenius}

\begin{proof}[Proof of Theorem~\ref{thm:frob-1}]
Without loss of generality, we argue in the case $p_0 = 0$. We also write $\norm{f} := [f]_0 + [\delta (\diff_{(g, x^d)} f)]_{\gamma}$ and $\norm{g} := [\delta g]_\beta$. With a slight abuse of notation we  write $g$ also to denote the function on $I^{m} \times \R^d$ given by $g(x^m, x^d) = g(x^m)$.  Let $J\subseteq I$ be an open interval with $0 \in J$, with length $|J| \le 1$ to be specified below, let $0<r\le 1$ also to be specified below and define
\[ V := \cur{ v \in \jet^{\alpha}_g(\bar J^m; \R^{d\times k}) \colon v_0 = f(0, \vartheta), [\delta v]_{\alpha} \le r },\]
with $\alpha \le \beta \gamma$ such that $\alpha + 2 \beta >2$ (although here may be possibly avoided, arguing with a more general $\alpha$ will be also useful in the proof of Theorem~\ref{thm:frob-2}).

For $v \in V$, define $F(v)\colon \bar{J}^m \to \R^{d\times k}$ by
\[ F( v )_p := f_{\bar{\theta}_p} = f_{\bra{p, \theta_p }}, \quad \text{for $p \in \bar J^m$,}\]
where $\theta_p := \vartheta +  \int_{[0 p]} v \cdot \d g$ and $\bar{\theta}_p := (p, \theta_p)$. Notice that, by~\eqref{eq:young}, one has the inequality
\begin{equation}\label{eq:norm-delta-theta-beta} [\delta \theta]_{\beta} \le \c(\alpha, \beta) [\delta g]_{\beta} [\delta v]_{\alpha} \le \c  \norm{g}_\beta  r \le \c \norm{g}_\beta.\end{equation}

The map  $F$ is well-defined, for Theorem~\ref{thm:jets} implies that $\theta \in C^{\beta}(J^m; \R^d)$ is $g$-differentiable with $\diff_g \theta = v$. By the chain rule (Proposition~\ref{prop:chain-rule}),  $F(v) \in C^{\beta}(J^m; \R^{d \times k})$ is $g$-differentiable, with
\begin{equation}\label{eq:proof-fixed-point-nablagF} \diff_g F(v)  = (\diff _g f)_{\bar \theta} + (\diff_{x^d} f)_{\bar \theta} v  \in C^{\alpha }(J^m; \R^{(dk)\times k}),\end{equation}
 hence by Proposition~\ref{prop:zust-1}, since $\alpha+ 2 \beta >2$, we deduce that $F(v) \in \jet^{\beta}_g(\bar J^m; \R^{d\times k})$.
%
  To show that $F$ is a contraction, if $|J|$ and $r$ are small enough, let $v$, $w \in V$ and write $\theta:= \vartheta+  \int v \d g$, $\varphi := \vartheta + \int v \d g$ so that, for $p\in J^m$,
  \[\begin{split}  F(v)_p - F(w)_p & = \delta f_{\bar{\theta}_p \bar{\varphi}_p}\\
  & = \int_{0}^1  (\diff_g f)_{u_t} \d (g\circ\bar{ u})_t +  \int_{0}^1 (\diff_{x^d} f)_{\bar u_t} \d (x^d\circ \bar{u})_t, \\
  & \quad \quad \text{with $\bar{u}_t := (1-t)\bar\theta_p + t\bar \varphi_p$, for $t \in [0,1]$}\\
  & =  \int_{0}^1 (\diff_{x^d} f)_{\bar u_t} \d (x^d \circ \bar u)_t, \quad \text{since $g( u_t) = p$ is constant,}\\
  & =  \bra{ \int_{0}^1 (\diff_{x^d} f)_{\bar u_t} \d t} \delta_{\theta_p \varphi_p} \quad \text{since $\delta (x^d \circ \bar u)_{st} = \delta_{\theta_p \varphi_p}  \delta_{st}$.}
  \end{split}\]

 Therefore, for $p$, $p' \in J^m$,  using Leibniz rule for $\delta$ and the additivity of the (Riemann) integral,
 \[\begin{split} \delta  ( F(v) - F(w)) _{pp'} & = \bra{ \int_{0}^1 (\diff_{x^d} f)_{\bar u_t'} -(\diff_{x^d} f)_{\bar u_t}\d t} \delta_{\theta_{p'} \varphi_{p'}} \\
 & \quad + \bra{ \int_{0}^1 (\diff_{x^d} f)_{\bar u_t} \d t} \delta (\varphi - \theta)_{pp'},\end{split}\]
 where $\bar u' := (1-t)\bar\theta_{p'} + t\bar \varphi_{p'}$, for $t \in [0,1]$. We bound separately the four terms in the right hand side above. First,
 \[ \begin{split} \abs{ \int_{0}^1 (\diff_{x^d} f)_{\bar u_t'} -(\diff_{x^d} f)_{\bar u_t}\d t} & \le [ (\diff_{x^d} f)_{\bar u} -(\diff_{x^d} f)_{\bar u'} ]_0\\
 & \le [\delta (\diff_{x^d} f)]_{\gamma} [\delta_{\bar{u},\bar{u'}}]_0^\gamma \quad \text{estimating the Riemann integral,} \\
 & = [\delta (\diff_{x^d} f)]_{\gamma} \bra{ |\delta_{pp'}| + \bra{ [\delta \theta]_\beta+[\delta \varphi]_\beta } |\delta_{pp'}|^\beta} ^{\gamma}, \\
 & \quad \quad \text{ since $\delta_{ \bar{u}_t \bar{u'}_t} = (\delta_{pp'}, (1-t) \delta_{\theta_p \theta_{p'} } + t\delta_{\varphi_p \varphi_{p'} })$,}\\
 & \le \norm{f} \bra{1 + 2 \c \norm{g} }^{\gamma} |\delta_{pp'}|^{\beta \gamma},
 \end{split}\]
  by~\eqref{eq:norm-delta-theta-beta}. Trivially, one has $\abs{\int_{0}^1 (\diff_{x^d} f)_{\bar u_t} \d t} \le [\diff_{x^d} f]_0$. 
    Using  ~\eqref{eq:integration-g-jet-2}, since $v_0 = w_0 = f(0, \vartheta)$, we have the inequality
     \[\begin{split} | \delta_{\theta_{p'} \varphi_{p'}} |   & = \abs{ \int_{[0,p']} (v-w) \d g} 
      \le \c(\alpha, \beta) [\delta (v-w)]_{\alpha} [\delta g]_\beta |J|^{\alpha+\beta}\\
      & \c \norm{g} |J|^{\alpha+\beta} [\delta (v-w)]_{\alpha}
     \end{split}\]
    and using~\eqref{eq:integration-g-jet-2},
     \[ \begin{split} |\delta (\varphi - \theta)_{pp'} |  & = \abs{ \int_{[pp']} (v-w) \d g} 
     \le \c(\alpha, \beta) [\delta g]_{\beta} [\delta (v-w)]_{\alpha} |\delta_{pp'}|^{\beta}\\
    &  \le \c \norm{g}   [\delta (v-w)]_{\alpha}|\delta_{pp'}|^{\beta}.\end{split}\]  
 Combining all these inequalities, we deduce that
   \[  \begin{split} | \delta  ( F(v) - F(w)) _{pp'} | & \le \c \norm{f} \norm{g} \bra{1 + 2 \c \norm{g} }^{\gamma}  |J|^{\alpha+\beta}  [\delta(v-w)]_{\alpha} |\delta_{pp'}|^{\beta \gamma}  \\
   & \quad  + \c \norm{f}  \norm{g} |J|^{\beta(1-\gamma)} [\delta (v-w)]_{\alpha} |\delta_{pp'}|^{\beta\gamma},\end{split}\]
   hence $F$ is a contraction if we choose $|J|$ and $r$ such that
  \begin{equation}\label{eq:condition-r-J-contraction-frob-1}   \c  \norm{f} \norm{g} \bra{ \bra{|J|^{1-\beta} + 2 \c \norm{g} }^{\gamma} |J|^{\alpha+ \beta}+ |J|^{\beta(1-\gamma)} } < 1. \qedhere \end{equation}
    \end{proof}

\begin{proof}[Proof of Theorem~\ref{thm:frob-2}]
We argue by induction over $m \ge 0$, the case $m=0$ being trivially true. Without loss of generality, we also assume that $p_0 = 0$. Assuming that the thesis holds for $m$, let $\vartheta\colon I^m \to \R^d$ solve~\eqref{eq:pfaff-germ} with $\vartheta_0=0$, and apply  Theorem~\ref{thm:g-differentiable-yde} with $(g^i)_{i=1}^{m}$ instead of $g$,  $m$ instead of $k$, $g^{m+1}$ instead of $y$ (and $h=1$), and  finally $f^{m+1}$ instead of $f$, obtaining a unique $(g^i)_{i=1}^{m+1}$-differentiable $\theta\colon I^{m+1} \to \R^d$ such that $\theta_{(p, 0)} = \vartheta_{p}$, for $p \in I^m$, $\diff_{(g^i)_{i=1}^{m+1}} \theta \in C^{\beta \gamma}(I^{m+1} ; \R^{d\times (m+1)})$ and, for $i \in \cur{1, \ldots, m}$,
\[ \begin{split} (\partial_{g^i} \theta)_{(p, t)} & =  (\partial_{g^i} \vartheta)_{(p, 0)} + \int_{[(p,0)(p,t)]}\bra{  (\partial_{g^i} f^{m+1})_{\bar{\theta}} +  (\diff_{x^d} f^{m+1})_{\bar{\theta}}\partial_{g^i} \theta } \d g^{m+1}\\
& =  f^i_{\bar{\vartheta}_{(p,0)}} +  \int_{[(p,0)(p,t)]}\bra{  (\partial_{g^i} f^{m+1})_{\bar{\theta}} +  (\diff_{x^d} f^{m+1})_{\bar{\theta}}\partial_{g^i} \theta } \d g^{m+1} \\
\end{split}
\]
by the inductive assumption. Using~\eqref{eq:involutivity}, one has
\[ (\partial_{g^{i}} f^{m+1})_{\bar{\theta}} = (\partial_{g^{m+1}} f^{i})_{\bar{\theta}} \]
and
\[ \begin{split} (\diff_{x^d} f^{m+1})_{\bar{\theta}}\partial_{g^i} \theta & = (\diff_{x^d} f^{m+1})_{\bar{\theta}} \bra{ \partial_{g^i} \theta  - f^{i}_{\bar{\theta}} } +  (\diff_{x^d} f^{m+1})_{\bar{\theta}} f^{i}_{\bar{\theta}}\\
& = (\diff_{x^d} f^{m+1})_{\bar{\theta}} \bra{ \partial_{g^i} \theta  - f^{i}_{\bar{\theta}} } + (\diff_{x^d} f^{i})_{\bar{\theta}} f^{m+1}_{\bar{\theta}},\end{split}\]
so that
\[ \begin{split} \int_{[(p,0)(p,t)]} & \bra{  (\partial_{g^i} f^{m+1})_{\bar{\theta}} +  (\diff_{x^d} f^{m+1})_{\bar{\theta}}\partial_{g^i} \theta } \d g^{m+1} \\&  = \int_{[(p,0)(p,t)]}\bra{(\diff_{x^d} f^{m+1})_{\bar{\theta}} + (\diff_{x^d} f^{m+1})_{\bar{\theta}} \bra{ \partial_{g^i} \theta  - f^{i}_{\bar{\theta}} } +(\diff_{x^d} f^{i})_{\bar{\theta}} f^{m+1}_{\bar{\theta}}}  \d g^{m+1}\\
& =  \int_{[(p,0)(p,t)]} \bra{ (\diff_{x^d} f^{m+1})_{\bar{\theta}} + (\diff_{x^d} f^{i})_{\bar{\theta}} f^{m+1}_{\bar{\theta}}} \\
& \quad + \int_{[(p,0)(p,t)]} (\diff_{x^d} f^{m+1})_{\bar{\theta}} \bra{ \partial_{g^i} \theta  - f^{i}_{\bar{\theta}} } \d g^{m+1} \\
& =  \delta f^i_{\bar{\theta}_{(p,0)} \bar{\theta}_{(p,t)}} + \int_{[(p,0)(p,t)]} (\diff_{x^d} f^{m+1})_{\bar{\theta}} \bra{ \partial_{g^i} \theta  - f^{i}_{\bar{\theta}} } \d g^{m+1},
\end{split}\]
since $\partial_{g^{m+1}} (f^i_{\bar{\theta}}) = (\diff_{x^d} f^{m+1})_{\bar{\theta}} + (\diff_{x^d} f^{i})_{\bar{\theta}} f^{m+1}_{\bar{\theta}}$. We conclude that the identity
\[ (\partial_{g^i} \theta)_{(p, t)} -f^i_{\bar{\theta}_{(p,t)}} = \int_{[(p,0)(p,t)]} (\diff_{x^d} f^{m+1})_{\bar{\theta}} \bra{ \partial_{g^i} \theta  - f^{i}_{\bar{\theta}} } \d g^{m+1}\]
holds. By Lemma~\ref{lem:gronwall}, with 
$$ a_t = (\partial_{g^i} \theta)_{(p, t)} -f^i_{\bar{\theta}_{(p,t)}}, \quad b_t = 0 \quad \text{and} \quad u_t = (\diff_{x^d} f^{m+1})_{\bar{\theta}_{(p,t)}},$$
since $a_0 = 0$ by inductive assumption, it follows that
\[ (\partial_{g^i} \theta)_{(p, t)} -f^i_{\bar{\theta}_{(p,t)}} = 0,\]
i.e., $\theta$ solves~\eqref{eq:pfaff}.
\end{proof}

\section{$g$-differentiability of Young differential equations}

The theorem below slightly extends known results on  (classical) differentiabily of solutions of Young differential equations (see e.g.~\cite{friz_multidimensional_2010}) to the case of $g$-differentiability.

\begin{theorem}[$g$-differentiability of YDE's]\label{thm:g-differentiable-yde}
Let $I \subseteq \R$, $J^m \subseteq \R^m$, let $\alpha, \beta, \gamma \in (0,1]$, with $\alpha \ge \beta$ and $\alpha+ \beta (1+\gamma)>2$. If 
\begin{enumerate}[(i)]
\item $y \in C^{\beta}( I; \R^{h})$, $g \in C^{\alpha}( J^m; \R^{k})$, 
\item $f\in C^{\beta}(   I\times J^m \times \R^d; \R^{d\times h } )$ is  $(y, g, x^d)$-differentiable with 
\[ (\diff_y f, \diff_{g} f, \diff_{x^d} f) \in C^{\gamma}(I \times J^m \times \R^d; \R^{(d h) \times (h+k+d)}),\]
\item $\vartheta \in C^{\alpha}(J^m; \R^d)$ is $g$-differentiable, with $\diff_g \vartheta \in C( J^m; \R^{d \times k})$,
\end{enumerate}
then, there exists a unique  $\theta\colon I\times J^m \to \R^d$ such that, 
\begin{equation}\label{eq:yde-g-diff} \theta_{(t,p)} = \vartheta_{p} + \int_{0}^t f_{(s,p, \theta_{(s,p)})} \d y_s = \vartheta_{p} + \int_{[(p,0) (p,t)]} f_{\bar{\theta}} \d y \quad \text{for every $p \in J^{m}$, $t \in I$,}\end{equation}
where $\bar{\theta}_{(t,p)} := ( t, p, \theta_{(t,p)})$. Moreover, $\theta$ is $(y,g)$-differentiable, with $\diff_y \theta = f_{\bar{\theta}}$, and 
\begin{equation}\label{eq:yde-g-deriv}  (\diff_g \theta)_{(t,p)} =  (\diff_g \vartheta)_{p} + \int_{[(0,p) (t,p)]}\bra{  (\diff_{g} f)_{\bar{\theta}} +  (\diff_{x^d} f)_{\bar{\theta}}\diff_g \theta } \d y \quad \text{
holds for  $t \in I$, $p \in J^m$.}\end{equation}
\end{theorem}

\begin{remark}\label{rem:alpha-beta-g-diff-yde}
Actually, the result holds provided that there exists an $x \in [0,1]$ such that the inequalities
\begin{equation}
\label{eq:true-condition-alpha-beta}
\alpha( x \gamma+1) >1 \quad \text{ and} \quad \beta( (1-x) \gamma+1)
\end{equation}
hold. Choosing $x = (1-\alpha)/\beta \gamma + \varepsilon$, for some $\varepsilon>0$ small enough so that $x \in [0,1]$, from  the inequality $\alpha+ \beta(1+\gamma)>2$ it follows that $\beta x\gamma+\alpha>1$  and $\beta( (1-x)\gamma+1)>1$. Since we also assume $\alpha \ge \beta$, then~\eqref{eq:true-condition-alpha-beta} hold. In the given form, the theorem encompasses both the case $\alpha = \beta$, so that the condition reduces to $\beta(2 + \gamma) >2$ and that of usual differentiability of Young differential equations, with $\alpha = 1$, so that the condition reduces to $\beta(1+ \gamma)>1$. 
\end{remark}

\begin{remark}[H\"older continuity of $\diff_g \theta$]\label{rem:hold-part-g-yde}
For any $\alpha' <  \alpha (\beta(1+\gamma) -1) /\beta$, if  $\diff _g \vartheta$ is $\alpha'$-H\"older continuous, the proof of Theorem~\ref{thm:g-differentiable-yde} yields that $p \mapsto (\diff_g \vartheta)_{(p, t)}$ is $\alpha'$-H\"older continuous as well. 
\end{remark}


Before we address the proof of this result, we prove a Gronwall-type inequality and a result on continuity of Young integrals. Although both results are known in the literature, we provide here statements  and self-contained proofs useful for our purposes.

\begin{lemma}[Young-Gronwall]\label{lem:gronwall}
Let $0 \in I \subseteq \R$, $\alpha$, $\beta \in (0,1]$, $\alpha+\beta>1$,
\[ a \in C^{\alpha}(I; \R^d), \quad b \in C^{\alpha}(I; \R^{d \times k}), \quad u \in C^{\alpha}(I; \R^{(d k) \times d}),  \quad \text{and} \quad y \in C^{\beta}(I; \R^k)\]
 be such that
\begin{equation} \label{eq:linear-yde} a_t =  a_0 + \int_{0}^t (b + u  a)  \d y \quad \text{for every  $t \in I$.}\end{equation}
Then, for some $\c = \c( \alpha, \beta, |I|, \norm{u}_{\alpha}, [\delta y]_\beta)$, one has
\begin{equation}\label{eq:gronwall-thesis} \norm{a}_{\beta} \le \c ( |a_0| + \norm{b}_{\alpha} ). \end{equation}
\end{lemma}

\begin{proof}
From~\eqref{eq:linear-yde} and~\eqref{eq:young} it follows that 
\[ \abs{ \delta a_{st} } = \abs{\int_s^t(b + ua) \d y_s} \le \tilde{\c} 
 \bra{\norm{ b} _\alpha + \norm{u}_\alpha \norm{a}_\alpha} [\delta y]_\beta  |\delta_{st}|^{\beta},\]
with, here and below, $\tilde\c := \c(\alpha, \beta) (1+|I|^\alpha)$, hence $\norm{a}_\beta<\infty$. To deduce~\eqref{eq:gronwall-thesis}, write
\[ \delta a_{st} = \int_{s}^t ( b + u a_0 ) \d y  +  \int_s^t u  (a - a_0) \d y,\]
and estimate separately the two terms, using again~\eqref{eq:young}. For the first term,
\[ \abs{ \int_{s}^t (  b  + u a_0) \d y} \le \tilde \c \bra{ \norm{b}_\alpha + \norm{u}_\alpha |a_0|} [\delta y]_\beta |\delta_{st}|^{\beta},\]
and for the second term,
\[\begin{split} \abs{ \int_s^t u  (a - a_0) \cdot \d y} & \le \tilde \c \norm{ u(a-a_0)}_\alpha [\delta y]_\beta  |\delta_{st}|^{\beta} \le  \tilde \c \norm{u}_\alpha [\delta a]_{\beta} |I|^{\beta-\alpha}[\delta y]_\beta  |\delta_{st}|^{\beta}\\
&  \le \frac 1 2  \norm{a}_\beta,
\end{split}\]
provided that $|I|$ is small enough so that the inequality
\begin{equation}\label{eq:condition-gronwall-small-I}\c(\alpha, \beta) (1+ |I|^{\alpha}) \norm{u}_\alpha |I|^{\beta-\alpha} [\delta y]_\beta \le \frac 1 2\end{equation}
holds. 
In such a case,~\eqref{eq:gronwall-thesis} holds with 
\begin{equation} \label{eq:c-small} \c :=  2  \tilde{\c}  (1+\norm{ u}_{\alpha}) [\delta y]_\beta.\end{equation}

For a general $I$, introduce a partition $\cur{ t_{-n} \le t_{-n+1} \ldots \le t_0 = 0 \le \ldots \le   t_n } \subseteq I$ such that, letting $I_i := [t_i, t_{i+1}]$,~\eqref{eq:condition-gronwall-small-I} holds with $|I_i|$ instead of $|I|$, which can be achieved with $n\ge 1$ depending on $\alpha$, $\beta$, $|I|$, $\norm{u}_{\alpha}$, $[\delta y]_\beta$ (over the entire interval $I$) only. Letting $\norm{a}_{\beta, i}$, $\norm{b}_{\alpha, i}$ denote respectively the $\beta$ and $\alpha$-H\"older norms of $a$ and $b$ restricted on each interval $I_i  = [t_i, t_{i+1}]$, for $i \in \cur{0, 1, \ldots, n-1}$, a straightforward induction gives that
\[ \norm{ a}_{\beta, i} \le \c^i |a_0| + \sum_{j=1}^{i} \c^{j-i} \norm{b}_{\alpha, i} \le n (1+\c)^n\bra{ |a_0| + \norm{b}_\alpha} \quad \text{for $i \in \cur{0,1, \ldots, n-1}$,}\]
with $\c$ as in~\eqref{eq:c-small}. Arguing similarly for $i \in \cur{0, -1, \ldots, -n+1}$, we obtain an analogous bound. Finally, given $s \in I_i$, $t \in I_j$, $s \neq t$, assuming without loss of generality that $i <j$,  we have
\[\begin{split} | \delta a_{st} | & \le | \delta a_{s t_{i+1}} | + \sum_{k=i+1}^{j-1} | \delta a_{t_k t_{k+1} } |  + | \delta a_{t_{j} t}|  \\
& \le \norm{a}_{\beta, i} |\delta_{s t_{i+1} }|^{\beta} + \sum_{k=i+1}^{j-1} \norm{a}_{\beta, k} | \delta_{t_k t_{k+1}} |^{\beta} +  \norm{a}_{\beta, j} |\delta_{t_j t }|^{\beta} \\
& \le \sum_{k=i}^{j} \norm{a}_{\beta, k}   |\delta_{st}|^{\beta} \le n^2(1+\c)^n\bra{ |a_0| + \norm{b}_\alpha},
\end{split}\]
hence the thesis with  the constant $n^2(1+\c)^n$, with $\c$ being as in~\eqref{eq:c-small}. 
%
%
\end{proof}



\begin{lemma}[Regularity of Young integral]\label{lem:fubini}
Let  $\alpha$, $\beta$, $\gamma \in (0,1]$ be such that $\alpha+ \beta>1$, $f\colon I \times J \to \R$ be such that
\[ \c_\alpha := \sup_{t \in J} \norm{ f(\cdot, t)}_{\alpha} < \infty, \quad   \c_{\gamma} := \sup_{s \in I} \norm{ f(s, \cdot)}_{\gamma}  < \infty \]
and let $g \in C^{\beta}(I)$. Then, for every $x \in (0,1]$ such that $x \alpha+ \beta>1$, there is $\c$ depending on $x$, $\alpha$, $\beta$, $\gamma$, $\c_{\alpha}$, $\c_{\gamma}$, $|I|$ and $|J|$  such that 
\[ \norm{ \int_I f (s, \cdot) \d g_s }_{(1-x)\gamma } \le \c [\delta g]_{\beta}. \]
\end{lemma}

\begin{proof}
Write $I = [s_0, s_1]$, $J= [t_0, t_1]$. For $t, t' \in J$, one has
\[  \begin{split} \abs{ \int_I f_{(s,t')} \d g_{s}  - \int_I f_{(s,t)} \d g_{s}} &  \le \abs{(f_{(s_0, t')} - f_{(s_0, t)}) \delta g_{s_0 s_1}} \\
&  \quad + \abs{   \int_{s_0}^{s_1} \bra{ f_{(s,t')}  -f_{(s,t)}} \d g_{s}  - (f_{(s_0, t')} - f_{(s_0, t)}) \delta g_{s_0 s_1}}\\
& \le [ \delta f(s_0, \cdot)]_{\gamma} | \delta_{tt'}|^{\gamma}  [\delta g]_{\beta} |I |^{\beta} \\
& \quad + \c [ \delta (f_{(\cdot,t')}- f_{(\cdot, t)} )]_{x \alpha} [\delta g]_{\beta} | I|^{x \alpha+ \beta} 
\end{split}\] by~\eqref{eq:young-germ}, with $\c = \c(x \alpha, \beta)$.
For $s$, $s' \in I$, we have
\[\begin{split} | \delta \bra{ f_{(\cdot,t')}  -f_{(\cdot,t) } }_{ss'} |  & = | f_{(s',t')} - f_{(s',t)}  -f_{(s,t')} + f_{(s,t)} | \\
& \le    \min \cur{ 2 \c_\alpha |\delta_{ss'} |^\alpha, 2 \c_\gamma |\delta_{tt'} |^\gamma } \le 2 \c_{\alpha}^{x} \c_\gamma^{1-x}  |\delta_{ss'} |^{x\alpha} |\delta_{tt'} |^{(1-x)\gamma},
\end{split}\]
thus
\[ [ \delta (f_{(\cdot,t')}- f_{(\cdot, t)} )]_{x \alpha} \le 2 \c_{\alpha}^{x} \c_\gamma^{1-x}  |\delta_{tt'} |^{(1-x)\gamma}. \]
As a consequence, we obtain the inequality
\[ \sqa{  \delta \bra{ \int_I f (s, \cdot) \d g_s} }_{(1-x)\gamma}   ( \c_{\gamma} |I|^{\beta}|J|^{(1-x)\gamma} + 2 \c \c_{\alpha}^x \c_\gamma^{1-x} | I|^{x\alpha + \beta}) [\delta g]\]
By \eqref{eq:young}, we also have 
\[ \abs{ \int_I f_{(s, t_0)} \d g _s } \le \c  \c_\alpha [\delta g]_\beta( |I|^{\alpha} + |I|^{\alpha+\beta}),\]
hence the thesis.
\end{proof}

\begin{proof}[Proof of Theorem~\ref{thm:g-differentiable-yde}]
The proof is split into four steps: first, we prove that $\theta$ is H\"older continuous. Then, letting $\diff_g \theta$ be defined as the solution to~\eqref{eq:yde-g-deriv}, we show that it is continuous. Finally, we prove $(g,y)$-differentiability of $\theta$.

Before addressing these points, we notice that, by the one-dimensional case of Theorem~\ref{thm:frob-1} both $\theta$ and $\diff_g \theta$ are uniquely determined respectively by~\eqref{eq:yde-g-diff} and~\eqref{eq:yde-g-deriv}, with
\[ \sup_{p \in J} \| \theta_{(\cdot, p)}\|_{\beta}+\| \diff_g \theta _{(\cdot, p)}\|_{\beta} =:\c_{\theta} < \infty.\] 

Given $p^0$, $p^1 \in J$, for $r \in [0,1]$ write  $p^r := (1-r)p^0 + r p^1$, $\theta^r_t := (1-r)\theta_{(t, p^0)} + r \theta_{(t, p^1)}$ for $t \in I$. Notice that, $r \mapsto p^r$ and $r \mapsto \theta^r_t$ are differentiable  (in the classical sense) with $\partial_r p^r = \delta_{p^0 p^1}$, $\partial_r \theta^r_t = \delta \theta_{(t, p^0) (t, p^1)}$. Moreover,
\[ |\delta_{p^r p^{r'}}|  = |r-r'| |\delta_{p^0 p^1}| \le |\delta _{p^0 p^1}|\quad \text{and}\quad \sup_{ r \in [0,1]} [ \delta \theta^r_{\cdot}]_{\beta} \le \c_{\theta}.\]

Throughout this proof we denote by $\c$ any constant (possibly varying from line to line)  depending upon $f$, $g$, $y$, $\vartheta$, $\alpha$, $\beta$, $\gamma$ and other parameters, but not upon $p_0$,  $p_1$ or $t \in I$.

\noindent{\emph{Step 1.\ (H\"older regularity of $\theta$)}}
To show that $|\theta_{(t, p^1)} - \theta_{(t,p^0)}| \le \c | \delta_{p^0 p^1}|^{\alpha}$, we apply Lemma~\ref{lem:gronwall} with $a_t := \theta_{(t,p^1)} - \theta_{(t,p^0)}$. Indeed, by~\eqref{eq:yde-g-diff}, we have the identity
\begin{equation}\label{eq:delta-theta-proof-yde-g-diff}  \begin{split}a_t & =  \theta_{(t,p^1)}  - \theta_{(t,p^0)} \\
& = \delta \vartheta_{p^0 p^1} +  \int_0^t \bra{ f_{ (s, p^1, \theta^1_s)} -  f_{ (s, p^0, \theta^0_s)}}  \d y_s \quad \text{by~\eqref{eq:yde-g-diff},}\\
& = \delta \vartheta_{p^0 p^1}   + \int_0^t \bra{ \int_0^1 \diff_g f_{(s, p^r, \theta^r_s)} \d g_{p^r} + \int_0^1 \diff_{x^d}  f_{(s,p^r, \theta^r_s)} \d r (  \theta_{(s, p^1)} - \theta_{(s,p^0)})  } \d y_s\\
& \quad  \text{by the chain rule Proposition~\ref{prop:chain-rule}, applied to $r \mapsto f_{(p^r, s, \theta^r_s)}$,}\\
& =  a_0  +  \int_0^t \bra{b_s + u_s a_s} \d y_s,
\end{split}\end{equation}
having defined, for $t \in I$, 
\[b_t := \int_0^1 \diff_g f_{(t,p^r,  \theta^r_t)} \d g_{p^r} \quad \text{and} \quad  u_t := \int_0^1 \diff_{x^d}  f_{(t, p^r, \theta^r_t)} \d r. \]
Since $\norm{\diff_g f}_{\gamma} < \infty$ and $r \mapsto (p^r, \theta^r_t)$ is differentiable,  by composition we obtain that
\begin{equation}\label{eq:sup-partial-gf-rough} \sup_{t \in I} [ \delta \diff_g f_{(t,p^{\cdot}, \theta^{\cdot}_t)} ]_{\gamma} \le \norm{\diff_g f}_{\gamma} \bra{|\delta_{p^0p^1}| + |\delta\theta_{(t,p^0)(t,p^1)}|}^{\gamma} \le \c.\end{equation}
On the other side,
\[ \sup_{r \in [0,1]} [ \delta \diff_g f_{(\cdot , p^{r}, \theta^{r}_\cdot )} ]_{\beta \gamma} \le \norm{\diff_g f}_{\gamma} \bra{|I|^{1 - \beta} + \c_\theta}^{\gamma} \le \c.\]
Moreover, $[\delta g_{p^\cdot}]_{\alpha} \le[ \delta g]_{\alpha} | \delta_{p^0p^1}|^\alpha$. Therefore, for any $x \in [0,1]$ such that $x \gamma + \alpha>1$, Lemma~\ref{lem:fubini} applied to $f_{(t,p^r, \theta^r_t)}$ and $g_{p^r}$ yields
\[ \norm{ b }_{(1-x) \beta \gamma} \le \c  |\delta_{p^0p^1}|^\alpha.\]
Similarly (in fact,  by standard properties of Riemann integral), $\norm{u}_{\gamma} \le \c$. If $(1-x)\beta \gamma + \beta >1$, the assumptions of Lemma~\ref{lem:gronwall} are satisfied, and we conclude that
\[ \sup_{t \in I} | a_t| \le  \norm{ a }_{\beta} \le \c \bra{| a_0| + \norm{ b}_{(1-x)\gamma} } \le \c | \delta_{p^0 p^1}|^{\alpha}.\]
To conclude, therefore, it is sufficient to notice that such a choice of $x \in [0,1]$ can be done, because of the assumption~$\alpha+ \beta(1+\gamma)>2$, (choosing $x$ slightly but strictly greater than $(1-\alpha)/\gamma$).

Notice that, a posteriori, we improve~\eqref{eq:sup-partial-gf-rough} to
\begin{equation}\label{eq:sup-partial-gf-improved} \sup_{t \in I} [ \delta \diff_g f_{(t, p^{\cdot}, \theta^{\cdot}_t)} ]_{\gamma} \le \norm{\diff_g f}_{\gamma} \bra{|\delta_{p^0p^1}| + |\delta\theta_{(t, p^0)(t,p^1)}|}^{\gamma} \le \c |\delta_{p^0p^1}|^{\alpha \gamma}.\end{equation} 
A similar inequality holds with $\diff_{x^d} f$ instead of $\diff_g f$.

\noindent{\emph{Step 2.\ (continuity of $\diff_g \theta$)}}
We apply again  Lemma~\ref{lem:gronwall}, in this case with $a_t := \diff_g \theta_{(t,p^1)} -\diff_g  \theta_{(t,p^0)}$. From~\eqref{eq:yde-g-deriv}, we have the identity
\[  \begin{split} a_t & = a_0 +  \int_0^t  \delta\bra{ \diff_g f_{\bar{\theta}_{(s,\cdot)}} + \diff_{x^d}f_{\bar{\theta}_{(s,\cdot)}} \diff_g \theta_{(s,\cdot)} }_{p^0p^1} \d y_s \\
& = a_0 + \int_0^t \bra{ \delta\bra{ \diff_g f_{\bar{\theta}_{(s,\cdot)}}}_{p^0p^1} + \bra{ \delta \diff_{x^d}f_{\bar{\theta}_{(s,\cdot)}}}_{p^0p^1}  \diff_g \theta_{(s,p^1)}  + f_{\bar{\theta}_{(s,p^0)} }a_s } \d y _s \quad 
\end{split}\]
using the discrete Leibniz rule~\eqref{eq:leibniz}. We then let
\[ b_t :=  \delta\bra{ \diff_g f_{\bar{\theta}_{(t,\cdot)}}}_{p^0p^1} + \bra{ \delta \diff_{x^d}f_{\bar{\theta}_{(t,\cdot)}}}_{p^0p^1}  \diff_g \theta_{(t,p^1)}, \quad \text{and} \quad  u_t :=f_{\bar{\theta}_{(t,p^0)} }.\] 
Clearly, $\norm{u}_{\beta \gamma} \le \c$. To estimate the H\"older norm of $b$, notice that by~\eqref{eq:sup-partial-gf-improved} we have $\|b\|_0 \le \c | \delta_{p^0 p^1}|^{\alpha \gamma}$, while by composition, $\norm{b}_{\beta \gamma} \le \c$ hence interpolating, for every $x \in [0,1]$,
\[\| b\|_{(1-x) \beta \gamma } \le  \c | \delta_{p^0 p^1}|^{x\alpha \gamma}.\]
Therefore, if $(1-x) \gamma \beta + \beta >1$, we obtain by Lemma~\ref{lem:gronwall} that
\[ \sup_{t \in I} |a_t| \le \norm{ a}_{\beta} \le \c \bra{ | a_0| + \norm{b}_{(1-x) \beta \gamma} } \le \c \bra{ \omega(\delta_{p^0 p^1}) + | \delta _{p^0 p^1}|^{x \alpha\gamma}},\]
$\omega$ denoting the modulus of continuity of $\diff_g \vartheta$. To obtain Remark~\ref{rem:hold-part-g-yde}, assuming that $\omega( \delta_{p^0 p^1}) \le \c | \delta_{p^0 p^1}|^{\alpha'}$ with $\alpha' < \alpha (\beta(1+\gamma) -1)/\beta$, it is sufficient to choose $x= \alpha' / \alpha \gamma$, so that $x \alpha \gamma = \alpha'$, and the condition $\beta((1-x) \gamma+1) >1$ is satisfied.

\noindent{\emph{Step 3.\ ($(y,g)$-differentiability of $\theta$)}} In this case, we apply  Lemma~\ref{lem:gronwall}, with $a_t := \theta_{(t,p^1)} - \theta_{(t,p^0)} - \diff_g \theta_{(t, p^0)} \delta g_{p^0 p^1}$. Taking the difference between~\eqref{eq:delta-theta-proof-yde-g-diff} and~\eqref{eq:yde-g-deriv} multiplied by $\delta g _{p^0 p^1}$, we obtain the identity
\[\begin{split} a_t &= a_0 + \int_0^t  \bra{ \int_0^1 \diff_g f_{(s,p^r, \theta^r_s)}  \d g_{p^r} -  \diff_g f_{(s,p^0, \theta^r_s)} \delta g_{p^0 p^1} } \d y_s \\
& \quad + \int_0^t \bra{\int_0^1 \diff_{x^d}  f_{(s,p^r, \theta^r_s)} \d r (  \theta_{(s,p^1)} - \theta_{(s,p^0)})   - \diff_{x^d}  f_{(s,p^0, \theta^0_s)} \diff_g \theta_{(s,p^0)}  \delta g_{p^0 p^1}  } \d y_s\\
& = a_0 + \int_0^t  \bra{ \int_0^1 \diff_g f_{(s,p^r, \theta^r_s)}  \d g_{p^r} -  \diff_g f_{(s,p^0, \theta^r_s)} \delta g_{p^0 p^1} } \d y_s\\
& \quad  + \int_0^t  \int_0^1 \diff_{x^d}\bra{  f_{(s,p^r,  \theta^r_s)} \d r - \diff_{x^d}  f_{(s,p^0,  \theta^0_s)}}\d r   \diff_g f_{(s,p^0, \theta^r_s)} \delta g_{p^0 p^1} \d y_s\\
& \quad  + \int_0^t \bra{\int_0^1 \diff_{x^d}  f_{(s,p^r,  \theta^r_s)} \d r} a_s \d y_s\\
& =: \int_0^t ( b_s^1 + b^2_s+ u_s a_s )\d y_s,
\end{split}\]
having defined
\[\begin{split} b_t^1 & := \int_0^1 \bra{ \diff_g f_{(t,p^r, \theta^r_t)} -  \diff_g f_{(t,p^0,  \theta^r_t)}}  \d g_{p^r},\\
   b^2_t &:=    \int_0^1 \bra{\diff_{x^d}  f_{(t,p^r, \theta^r_t)}- \diff_{x^d}  f_{(t,p^0,  \theta^0_t)}  }\d r  \diff_g f_{(t,p^0, \theta^0_t)} \delta g_{p^0 p^1} ,\quad   \text{and}\\ 
   u_t & :=\int_0^1 \diff_{x^d}  f_{(t,p^r,  \theta^r_t)} \d r.\end{split}\]
   
 By composition, one has that $\norm{u}_{\beta \gamma} \le \c$. We prove below that, for every $x \in [0,1]$,
\begin{equation}\label{eq:yde-diff-claim1} \|b^1\|_{(1-x) \beta \gamma} \le \c | \delta_{p^0 p^1}|^{\alpha(x \gamma+1)} \quad \text{if $x \gamma+\alpha>1$,}\end{equation}
and 
\begin{equation}\label{eq:yde-diff-claim2}  \|b^2\|_{(1-x) \beta \gamma} \le \c  | \delta_{p^0 p^1} |^{\alpha\bra{ x\gamma +1}}\end{equation}
By Remark~\ref{rem:alpha-beta-g-diff-yde}, we can choose $x \in [0,1]$ such that $\alpha(x\gamma+1)$ and $\beta\bra{(1-x) \gamma+1}>1$, so that in particular $x \gamma+ \alpha>1$ and~\eqref{eq:yde-diff-claim1} holds and Lemma~\ref{lem:gronwall} applies (with $(1-x)\beta \gamma$ instead of $\alpha$). We obtain
 \[ \norm{a}_0 \le \norm{a}_{\beta} \le \c \bra{ |a_0| + \|b^1+b^2\|_{(1-x) \beta \gamma}} \le \c \bra{o(\delta_{p^0 p^1})+ | \delta_{p^0 p^1}|^{\alpha\bra{ x\gamma +1}}} = o(\delta_{p^0 p^1}).\]
This proves that, uniformly with respect to $t \in I$, $p \mapsto \theta_{(t,p)}$ is $g$-differentiable with $g$-derivative $\diff_g \theta_{(t,p)}$. By construction, $t \mapsto \theta_{(p, t)}$ is $y$-differentiable with $\partial_y \theta_{(p,t)}=  f_{\bar{\theta}_{(t,p)}}$. An application of the triangle inequality yields then that $\theta$ is $(y,g)$-differentiable.

\noindent{\emph{Step 4.\ Proof of~\eqref{eq:yde-diff-claim1} and~\eqref{eq:yde-diff-claim2}}.}

By composition, one has the inequality
   \[ \sup_{r \in [0,1]}   \norm{  \diff_g f_{(\cdot, p^{r},  \theta^{r}_\cdot)} -  \diff_g f_{(\cdot, p^0,  \theta^{0}_\cdot)}}_{\beta \gamma} \le \c.\]
By~\eqref{eq:sup-partial-gf-improved}, we have also
   \[ \sup_{t \in I} \norm{  \diff_g f_{(t, p^{\cdot}, \theta^{\cdot}_t)} -  \diff_g f_{(t, p^0,  \theta^{0}_t)}}_{\gamma} =   \sup_{t \in I} [ \delta \diff_g f_{(t, p^{\cdot}, \theta^{\cdot}_t)} ]_{\gamma} \le \c | \delta_{p^0 p^1}|^{\alpha \gamma}.\]
Since $[ \delta g_{p^\cdot}]_{\alpha} \le \c | \delta_{p^0 p^1 }|^{\alpha}$, Lemma~\ref{lem:fubini} entails the validity of~\eqref{eq:yde-diff-claim1}. 
To prove~\eqref{eq:yde-diff-claim2}, we argue similarly with the Riemann integral: by composition and standard properties of Riemann integral, one has
\[ \norm{ \int_0^1\bra{ \diff_{x^d}  f_{(\cdot, p^r,  \theta^r_\cdot)} -  \diff_{x^d}  f_{(\cdot, p^0,  \theta^0_\cdot)} } \d r}_{\beta \gamma} \le  \c.\]
By~\eqref{eq:sup-partial-gf-improved}, with $\diff_{x^d}f$ instead of $\diff_g f$, it follows that
\[ \sup_{t \in I} \abs{ \int_0^1\bra{ \diff_{x^d}  f_{(t, p^r, \theta^r_t)} -  \diff_{x^d}  f_{(t, p^0,  \theta^0_t)} } \d r} \le \c | \delta_{p^0 p^1}|^{\alpha \gamma},\]
Interpolating, for every $x \in [0,1]$,
\[  \norm{ \int_0^1\bra{ \diff_{x^d}  f_{(\cdot, p^r,  \theta^r_\cdot)} -  \diff_{x^d}  f_{(\cdot, p^0,  \theta^0_\cdot)} } \d r}_{(1-x)\beta \gamma} \le  \c| \delta_{p^0 p^1}|^{x\alpha \gamma}.\]
Since $\norm{ \partial_{g} f_{(\cdot, p^0,  \theta^0_\cdot)}}_{\beta \gamma} \le \c$ and $| \delta g_{p^0 p^1}| \le \c | \delta_{p^0 p^1}|^{\alpha}$, we conclude that~\eqref{eq:yde-diff-claim2} holds as well. 
\end{proof}

\section{Proof of Proposition~\ref{prop:compensated1}}\label{sec_proof_compensated1}
Let $F$ be the dyadically additive function defined on rectangles $Q \subseteq I^2$ by
\[  F(Q) := \int_{\partial Q} (\theta_1 -\mathfrak{v}) \d g_1 + \theta_2 \d g_2.\]
We argue that $F(Q) = o( \diam(Q)^2)$ as $\diam(Q) \to 0$. 

For fixed $Q \subseteq I^2$ and $\bar{p} \in Q$, we notice that we can always assume that $g_{\bar{p}} = 0$, and restrict from $I^2$ to $Q$, so that the inequality $[g]_0 \le [\delta g]_\beta \diam(Q)^\beta$  holds. 
Integrating by parts, we have
\begin{equation} \label{eq:proof-key-lemma-2} F(Q) = -\int_{\partial Q}  g_1 \d \theta_1 -  \int_{\partial Q} g_2 \d \theta_2 -\int_{\partial Q} \mathfrak{v} \d g_1. \end{equation}
Arguing as in the proof of~\eqref{eq:stokes-without-zust}, using the fact that $\theta_1$ in  $g$-differentiable, one has the identity
\[ \int_{\partial Q}  g_1 \d \theta_1    =  \int_{\partial Q}  g_1 (\partial_{g_1} \theta_1) \d g_1 +\int_{\partial Q}  g_1 (\partial_{g_2} \theta_1) \d g_2 \]
and, integrating by parts, the first integral can be estimated as 
\[ \int_{\partial Q} g_1 ( \partial _{g_1} \theta) \d g_1 =  \frac 1 2 \int_{\partial Q} (\partial_{ g_1} \theta) \d g_1^2 = O \bra{\diam(Q)^{\alpha+2 \beta}}.\]
For the second integral, we first integrate by parts,
\[ \int_{\partial Q}  g_1 (\partial_{g_2} \theta_1) \d g_2 = - \int_{\partial Q} g_2 \d (g_1  (\partial_{g_2} \theta_1)) = -  \int_{\partial Q} g_2 g_1 \d( \partial_{g_2} \theta_1) - \int_{\partial Q} g_2 ( \partial_{g_2} \theta_1 ) \d g_1,\]
and then estimate the first integral using~\eqref{eq:young-boundary},
\[ \begin{split}  \abs{ \int_{\partial Q} g_1 g_2 \d( \partial_{g_2} \theta_1)}&  \le \c [\delta (g_1 g_2)]_\beta [ \delta  (\partial_{g_2} \theta_1)]_\alpha \diam(Q)^{\alpha+\beta}\\
& \le  2 \c [\delta g_1]_\beta [\delta g_2]_\beta [ \delta  (\partial_{g_2} \theta_1)]_\alpha \diam(Q)^{\alpha+2\beta } =o\bra{\diam(Q)^2},\end{split}\]
using the inequality $[\delta(g_1 g_2)]_\beta \le [g_1]_0 [\delta g_2]_\beta+  [\delta g_1]_\beta[g_2]_0 \le 2 [\delta g_1]_\beta \delta g_2]_\beta \diam(Q)^\beta$.

 Arguing as in the proof of ~\eqref{eq:stokes-without-zust}, we have the following identity for the second integral in~\eqref{eq:proof-key-lemma-2}:
 \[ \int_{\partial Q}  g_2 \d \theta_2    =  \int_{\partial Q}  g_2 (\partial_{g_1} \theta_2) \d g_1 +\int_{\partial Q}  g_2 (\partial_{g_2} \theta_2) \d g_2  \] 
and the second integral above be shown to be $o\bra{\diam(Q)^2}$. Putting all these facts together, we see that the thesis amounts to prove that
\begin{equation}\label{eq:before-fubini}  \int_{\partial Q} \bra{ h g_2  - \mathfrak{v}}  \d g_1  = o\bra{ \diam(Q)^{2}},\end{equation}
where for brevity we write $h :=  \partial_{g_1} \theta_2 - \partial_{g_2} \theta_1  \in C^{\alpha}(I^2)$. 

By definition of integral along the boundary $\partial Q$ and the fact that $g_1 = g_1(s)$ depends on the variable $s$ only, we have that, writing $Q = [s_0, s_1]\times[t_0, t_1]$,
\begin{equation}\label{eq:fubini-two-integrals} \begin{split} \int_{\partial Q} \bra{ h g_2  - \mathfrak{v}}  \d g_1 & = \int_{[(s_0, t_0) (s_1, t_0)]}  (h g_2 - \mathfrak{v} ) \d g _1+ \int_{[(s_1, t_1)(s_0, t_1) ]}  ( h g_2 - \mathfrak{v} ) \d g_1\\
& =- \int_{s_0}^{s_1} \delta (h g_2  - \mathfrak{v})_{(s,t_0)(s,t_1)} \d g_1(s)\\
& = - \int_{s_0}^{s_1}\delta h_{(s,t_0)(s,t_1)} g_2(s, t_1) \d g_1(s) \\
& \quad \quad - \int_{s_0}^{s_1} \bra{ h_{(s,t_0)} (\delta g_2)_{(s,t_0)(s,t_1)} - \delta \mathfrak{v}_{(s,t_0)(s,t_1)}} \d g_1(s).
\end{split}\end{equation}
To conclude, we show separately that both integrals in the last two lines above are $o(\diam(Q)^2)$ as $\diam(Q)\to 0$. For the first one, by applying~\eqref{eq:young} on $I= [s_0, s_1]$, with $\gamma:=\min\cur{\alpha, \beta}$ instead of $\alpha$ (notice that $\alpha+2 \beta >2$ implies $\gamma+\beta>1$), $f_s :=\delta h_{(s,t_0)(s,t_1)} g_2(s, t_1)$ and $g=g_1$, hence we see that it is sufficient to prove the inequalities
\begin{equation}\label{eq:f-0-fubini} [f]_0 = \sup_{s \in [s_0, s_1]} \abs{ f_{s}} = O\bra{  \diam(Q)^{\alpha+\beta}}\end{equation}
and
\begin{equation} \label{eq:f-alpha-fubini}\quad [\delta f]_{\gamma} = \sup_{\substack{s, s' \in [s_0, s_1] \\ s \neq s' }} \frac{ \abs{ \delta f_{ss'} }} { |\delta_{ss'}|^{\gamma} } = O\bra{ \diam(Q)^{\alpha + \beta-\gamma}}.
\end{equation}
Indeed, we have for $s \in [s_0, s_1]$,
\[ |f_{s}| \le \abs{ \delta h_{(s,t_0)(s,t_1)} g_2(s, t_1)} \le [\delta h]_{\alpha} [g_2]_0 |\delta_{t_0 t_1}|^\alpha \le [\delta h]_{\alpha} [\delta g_2]_\beta \diam(Q)^{\alpha+\beta},\]
and for $s$, $s' \in [s_0, s_1]$,
\[  \begin{split} \abs{ (\delta f_2)_{ss'} } & \le  \abs{ \delta h_{(s',t_0) (s', t_1) }} \abs{ (\delta g_2)_{(s, t_1) (s',t_1)} } + \abs{  g_2(s, t_1 ) } \bra{ \abs{ \delta h_{(s, t_1) (s', t_1)} } +\abs{ \delta h_{(s, t_0) (s', t_0)} }} \\
& \le [\delta h]_{\alpha} [\delta g_2]_\beta  |\delta_{ss'}|^\beta |\delta_{t_0t_1}| ^{\alpha} + 2 [g_2]_0 [\delta h]_\alpha |\delta_{ss'}|^{\alpha} \\
& \le 3 [\delta g_2]_\beta [\delta h]_{\alpha} |\delta_{ss'}|^{\gamma} \diam(Q)^{\alpha + \beta- \gamma}.\end{split}\]

For the second integral in~\eqref{eq:fubini-two-integrals},  by applying~\eqref{eq:young} on $I= [s_0, s_1]$, with $\gamma:=\alpha/2$ instead of $\alpha$ (notice that $\alpha+2 \beta >2$ implies $\gamma+\beta>1$),
\[ f_s := h_{(s,t_0)} (\delta g_2)_{(s,t_0)(s,t_1)} - \delta \mathfrak{v}_{(s,t_0)(s,t_1)} = -\int_{t_0}^{t_1} \delta h_{(s,t_0)(s,t)} \d g_2(s, \cdot) (t)\]
 and $g=g_1$, the thesis follows if we prove the analogues of~\eqref{eq:f-0-fubini} and~\eqref{eq:f-alpha-fubini} (with $\gamma = \alpha/2$). Indeed, for $s \in [s_0, s_1]$,~\eqref{eq:young-germ} implies
\[ |f_{s}| \le  \c(\alpha, \beta) [\delta h]_{\alpha} [\delta g_2]_{\beta} |\delta_{t_0 t_1}|^{\alpha+\beta}  \le \c(\alpha, \beta) [\delta h]_{\alpha} [\delta g_2]_{\beta} \diam(Q)^{\alpha+ \beta},\]
while for $s$, $s' \in [s_0, s_1]$,  we furthermore decompose $\delta f_{ss'}$ via  the identity
\begin{equation}\label{eq:final-fubini}\begin{split} (\delta f_2)_{ss'} & = \int_{t_0}^{t_1} \delta h_{(s',t_0)(s', t)} \d g_2(s', \cdot)(t)  - \int_{t_0}^{t_1} \delta h_{(s,t_0)(s, t)} \d g_2(s, \cdot)(t) \\
& =  \int_{t_0}^{t_1} \bra{ \delta h_{(s',t_0)(s', t)} -  \delta h_{(s,t_0)(s, t)} } \d g_2(s', \cdot)(t)  \\
& \quad + \int_{t_0}^{t_1} \delta h_{(s, t_0)(s,t)} \d \bra{ g_2(s, \cdot) - g_2(s',\cdot)}(t)
\end{split}\end{equation}
and estimate separately the two integrals, using~\eqref{eq:young-germ}  (indeed, both ``integrands'' for $t = t_0$ are null). In the first integral, we use the pair $(\gamma, \beta)$ instead of $(\alpha, \beta)$, obtaining, for some $\c = \c(\gamma, \beta)$, 
\[ \begin{split} & \abs{  \int_{t_0}^{t_1} \bra{ \delta h_{(s',t_0)(s', t)} -  \delta h_{(s,t_0)(s, t)} } \d g_2(s', \cdot)(t) }  \\ 
& \quad \quad  \le \c [ \delta( \delta h_{(s',t_0)(s', \cdot)} -  \delta h_{(s,t_0)(s, \cdot)} ) ]_{\gamma} [\delta g_2]_\beta |\delta_{t_0 t_1}|^{\gamma+ \beta} \\
&  \quad \quad \le 2 \c [\delta h]_\alpha [\delta g_2]_\beta \diam(Q)^{\gamma+ \beta} |\delta_{s, s'}|^{\gamma},
\end{split}\]
where we used the bound $[ \delta( \delta h_{(s',t_0)(s', \cdot)} -  \delta h_{(s,t_0)(s, \cdot)} ) ]_{\gamma} \le 2 ]\delta h]_\alpha |\delta_{s, s'}|^{\gamma}$, that follows from the inequality, for $t$, $t' \in [t_0, t_1]$, 
\[\begin{split} \abs{ \delta( \delta h_{(s',t_0)(s', \cdot)} -  \delta h_{(s,t_0)(s, \cdot)} )_{tt'}  }&  =  \abs{ h_{s', t'} - h_{s', t} - h_{s, t'} + h_{s, t} } \\
& \le 2 [\delta h]_{\alpha} \min\cur{ |\delta_{ss'}|^{\alpha},|\delta_{tt'}|^{\alpha} } \le   2 [\delta h]_{\alpha}|\delta_{ss'}|^{\gamma} |\delta_{tt'}|^{\gamma}.\end{split}\]
For the second integral in~\eqref{eq:final-fubini}, we use the pair $(\alpha, \beta-\gamma)$ instead of $(\alpha, \beta)$ (notice that $\alpha+2\beta>2$ implies $\beta>1/2$, hence $\beta> \alpha/2 = \gamma$), obtaining for some $\c = \c(\alpha, \beta - \gamma)$,
\[ \begin{split} & \abs{ \int_{t_0}^{t_1} \delta h_{(s, t_0)(s,t)} \d \bra{ g_2(s, \cdot) - g_2(s',\cdot)}(t)}
\\  &\quad \quad  \le 
\c [ \delta( \delta h_{(s,t_0)(s, \cdot)})]_{\alpha} [\delta ( g_2(s, \cdot) - g_2(s',\cdot) )]_{\beta- \gamma} |\delta_{t_0 t_1}|^{\gamma+\beta}\\
&  \quad \quad \le  \c [\delta h]_{\alpha} [\delta g_2]_\beta \diam(Q)^{\gamma+ \beta}  |\delta_{ss'}|^{\gamma} 
\end{split}\]
where we used the bound  $[\delta ( g_2(s, \cdot) - g_2(s',\cdot) )]_{\beta- \gamma} \le [\delta g_2]_\beta  |\delta_{ss'}|^{\gamma} $ that follows from the inequality, for $t$, $t' \in [t_0, t_1]$, 
\[ \begin{split} \abs{ \delta ( g_2(s, \cdot) - g_2(s',\cdot)_{tt'}} & =   \abs{ g_2(s, t') - g_2(s, t) + g_2(s',t') -g_2(s',t) } \\
& \le 2 [\delta g_2]\min\cur{ |\delta_{ss'}|^{\beta},|\delta_{tt'}|^{\beta} } \le  2 [\delta g_2] |\delta_{ss'}|^{\gamma} |\delta_{tt'}|^{\beta - \gamma}. \qedhere
\end{split} \]

\bibliographystyle{unsrt}

\begin{thebibliography}{10}

\bibitem{hartman_ordinary_2002}
P.~Hartman.
\newblock {\em Ordinary {Differential} {Equations}}.
\newblock Classics in {Applied} {Mathematics}. Society for Industrial and
  Applied Mathematics, January 2002.
\newblock DOI: 10.1137/1.9780898719222.

\bibitem{simic_1996}
Slobodan Simi{\'c}.
\newblock Lipschitz distributions and anosov flows.
\newblock {\em Proceedings of the American Mathematical Society},
  124(6):1869--1877, 1996.

\bibitem{rampazzo_2007}
F.~{Rampazzo} and H.~J. {Sussmann}.
\newblock {Commutators of flow maps of nonsmooth vector fields}.
\newblock {\em Journal of Differential Equations}, 232:134--175, 2007.

\bibitem{montanari_frobenius-type_2013}
Annamaria Montanari and Daniele Morbidelli.
\newblock A {Frobenius}-type theorem for singular {Lipschitz} distributions.
\newblock {\em Journal of Mathematical Analysis and Applications},
  399(2):692--700, March 2013.

\bibitem{luzzatto_integrability_2016}
Stefano Luzzatto, Sina Tureli, and Khadim War.
\newblock Integrability of continuous bundles.
\newblock {\em Journal für die reine und angewandte Mathematik (Crelles
  Journal)}, 0(0), 2016.

\bibitem{friz_course_2014}
Peter~K. Friz and Martin Hairer.
\newblock {\em A course on rough paths}.
\newblock Universitext. Springer, Cham, 2014.

\bibitem{gubinelli_controlling_2004}
M.~Gubinelli.
\newblock Controlling rough paths.
\newblock {\em J. Funct. Anal.}, 216(1):86--140, 2004.

\bibitem{magnani_2018}
Valentino Magnani, Eugene Stepanov, and Dario Trevisan.
\newblock A rough calculus approach to level sets in the heisenberg group.
\newblock {\em Journal of the London Mathematical Society}, 97(3):495--522,
  2018.

\bibitem{pansu_1989}
Pierre Pansu.
\newblock Metriques de carnot-caratheodory et quasiisometries des espaces
  symetriques de rang un.
\newblock {\em Annals of Mathematics}, 129(1):1--60, 1989.

\bibitem{stepanov_sewing_17}
Eugene Stepanov and Dario Trevisan.
\newblock Towards geometric integration of rough differential forms.
\newblock 2017.

\bibitem{young_inequality_1936}
L.~C. Young.
\newblock An inequality of the {Hölder} type, connected with {Stieltjes}
  integration.
\newblock {\em Acta Math.}, 67(1):251--282, 1936.

\bibitem{zust_integration_2011}
Roger Z\"{u}st.
\newblock Integration of {H\"older} forms and currents in snowflake spaces.
\newblock {\em Calc. Var. Partial Differential Equations}, 40(1-2):99--124,
  2011.

\bibitem{zust_2015}
Roger Z\"{u}st.
\newblock Some results on maps that factor through a tree.
\newblock {\em Anal. Geom. Metr. Spaces}, 3(1):73--92, 2015.

\bibitem{brezis_2011}
Ha\"{\i}m Brezis and Hoai-Minh Nguyen.
\newblock The {J}acobian determinant revisited.
\newblock {\em Invent. Math.}, 185(1):17--54, 2011.

\bibitem{filipovic_2001}
Damir Filipovi\'{c}.
\newblock {\em Consistency problems for {H}eath-{J}arrow-{M}orton interest rate
  models}, volume 1760 of {\em Lecture Notes in Mathematics}.
\newblock Springer-Verlag, Berlin, 2001.

\bibitem{alberti_2014}
Giovanni Alberti, Stefano Bianchini, and Gianluca Crippa.
\newblock A uniqueness result for the continuity equation in two dimensions.
\newblock {\em J. Eur. Math. Soc. (JEMS)}, 16(2):201--234, 2014.

\bibitem{bianchini_2016}
S.~Bianchini and N.~A. Gusev.
\newblock Steady nearly incompressible vector fields in two-dimension: chain
  rule and renormalization.
\newblock {\em Arch. Ration. Mech. Anal.}, 222(2):451--505, 2016.

\bibitem{wenger2018constructing}
Stefan Wenger and Robert Young.
\newblock Constructing h{\"o}lder maps to carnot groups.
\newblock 2018.

\bibitem{feyel_curvilinear_2006}
Denis Feyel and Arnaud de~La~Pradelle.
\newblock Curvilinear integrals along enriched paths.
\newblock {\em Electron. J. Probab.}, 11:no. 34, 860--892, 2006.

\bibitem{friz_multidimensional_2010}
Peter~K. Friz and Nicolas~B. Victoir.
\newblock {\em Multidimensional stochastic processes as rough paths}, volume
  120 of {\em Cambridge {Studies} in {Advanced} {Mathematics}}.
\newblock Cambridge University Press, Cambridge, 2010.

\end{thebibliography}

\end{document}